\documentclass[a4paper,twoside,11pt]{article}

\usepackage{a4wide, fancyhdr, amsmath, amssymb, mathtools, yfonts, enumerate}
\usepackage{mathrsfs}
\usepackage{graphicx}
\usepackage{tikz-cd}
\usepackage[all]{xy}
\usepackage[utf8]{inputenc}
\usepackage{amsthm}
\usepackage[english]{babel}
\usepackage{chngcntr}
\usepackage{ifthen}
\usepackage{calc}
\usepackage{hyperref}
\usepackage{authblk}
\numberwithin{equation}{section}
\counterwithin*{equation}{section}
\usepackage{makecell}

\newcommand{\Z}{\mathbb{Z}}
\newcommand{\Q}{\mathbb{Q}}
\newcommand{\R}{\mathbb{R}}
\newcommand{\C}{\mathbb{C}}

\newcommand\FF{\mathbb{F}}

\newcommand\sB{\mathscr{B}}

\newcommand\wO{\widehat{\Omega}}

\newcommand{\PPP}{\mathbb{P}}
\newcommand{\AAA}{\mathbb{A}}
\newcommand{\leg}[2]{\left(\frac{#1}{#2}\right)}

\newcommand{\tw}{\mathrm{tw}}
\newcommand{\Sub}{\mathrm{Sub}}
\newcommand{\FAlt}{\mathrm{BM}_{\Sub}}
\newcommand{\OBMS}{\mathrm{OBM}_{\Sub}}
\newcommand{\BrM}{\mathrm{BrM}}
\newcommand{\OAlt}{\mathrm{OAlt}}
\newcommand{\RAlt}{\mathrm{RAlt}}
\newcommand{\PAlt}{\mathrm{PBM}}
\newcommand{\LRS}{\mathrm{LRS}}
\newcommand{\Jay}{\mathrm{BM_1}}

\DeclareMathOperator{\rk}{rk}

\DeclareMathOperator{\Br}{Br}
\DeclareMathOperator{\sgn}{sgn}
\DeclareMathOperator{\Alt}{Alt}
\DeclareMathOperator{\Pic}{Pic}

\DeclareMathOperator{\inv}{inv}
\DeclareMathOperator{\Li}{Li}
\DeclareMathOperator{\Tw}{T}
\DeclareMathOperator{\Ind}{Ind}
\DeclareMathOperator{\Res}{Res}
\DeclareMathOperator{\Het}{H}

\newtheorem{lemma}{Lemma}[section]
\newtheorem{theorem}[lemma]{Theorem}
\newtheorem{prop}[lemma]{Proposition}
\newtheorem{corollary}[lemma]{Corollary}
\newtheorem{definition}[lemma]{Definition}
\newtheorem{conjecture}[lemma]{Conjecture}
\newtheorem{remark}[lemma]{Remark}

\theoremstyle{definition}
\newtheorem{example}[lemma]{Example}

\title{\vspace{-\baselineskip}\sffamily\bfseries Serre's problem for multiple conics}
\author[1]{Stephanie Chan}
\author[2]{Peter Koymans}
\author[3]{Nick Rome}
\affil[1]{IST Austria}
\affil[2]{Utrecht University}
\affil[3]{TU Graz \vspace*{-0.5\baselineskip}}

\date{\today}

\begin{document}
\maketitle

\vspace*{-\baselineskip}

\begin{abstract}
We prove the refined Loughran--Smeets conjecture of Loughran--Rome--Sofos for a wide class of varieties arising as products of conic bundles. One interesting feature of our varieties is that the subordinate Brauer group may be arbitrarily large. 

As an application of our methods, we answer a question of Lenstra by giving an asymptotic for the triples of integers $(a, b, c)$ for which the R\'edei symbol $[a, b, c]$ takes a given value. We also make significant progress on a question of Serre on the zero loci of systems of quaternion algebras defined over $\Q(t_1, \dots, t_n)$.
\end{abstract}

\vspace*{-\baselineskip}

\setcounter{tocdepth}{2}
\tableofcontents

\section{Introduction}
\subsection{Solubility in families}
In 1990, Serre~\cite{Ser90} considered how frequently one can find non-trivial solutions to conics in certain natural families. He produced a conjecturally sharp upper bound for the number of soluble fibres in these families and in some cases lower bounds. Recently, Loughran--Rome--Sofos~\cite{LRS} have extended Serre's bounds to an asymptotic formula for the critical example of the family of planar diagonal conics. Our main result is a dramatic generalisation allowing us to produce an asymptotic formula for the frequency of soluble fibres in a large class of products of conic bundles.

In particular, let $n \in \Z_{\geq 2}$, $m \in \Z_{\geq 1}$ and let $M_{i, j}(t_1, \dots, t_n)$ be squarefree monomials with leading coefficient $\pm 1$ for $i =1,2,3$ and $j = 1, \ldots, m$. We also assume that $M_{i, 1}, M_{i, 2}, M_{i, 3}$ are pairwise coprime and have the same degree. Consider the system of simultaneous equations in $3m$ variables $x_1, \ldots, x_m$, $y_1, \ldots, y_m$, $z_1, \ldots, z_m$ given by 
\begin{equation}\label{eq:def}
\begin{alignedat}{4}
&M_{1, 1}(t_1, \dots, t_n) x_1^2 \, &&+ \, M_{1, 2}(t_1, \dots, t_n) y_1^2 \, && \, + M_{1, 3}(t_1, \dots, t_n) z_1^2 &&= 0 \\
&M_{2, 1}(t_1, \dots, t_n) x_2^2 \, &&+ \, M_{2, 2}(t_1, \dots, t_n) y_2^2 \, && \, + M_{2, 3}(t_1, \dots, t_n) z_2^2 &&= 0\\
& \quad \quad \quad \quad\ \vdots && \quad \quad \quad \quad \quad \ \ \vdots && \quad \quad \quad \quad \quad \ \ \vdots && \ \ \vdots \\
&M_{m, 1}(t_1, \dots, t_n) x_m^2 \, &&+ \, M_{m, 2}(t_1, \dots, t_n) y_m^2 && \, + M_{m, 3}(t_1, \dots, t_n) z_m^2 &&= 0.
\end{alignedat}
\end{equation}
For any real $B \geq 1$, denote the counting function
\[
N^{\mathrm{proj}}(B)
\coloneqq
\#\{t \in \PPP^{n - 1}(\Q) : H(t) \leq B,~\eqref{eq:def} \text{ has non-trivial } \Q \text{ solutions}\},
\]
where $H$ is the anticanonical height on $\PPP^{n - 1}_{\Q}$ (we shall henceforth drop the subscript from $\PPP^{n - 1}_{\Q}$ and $\AAA^{n - 1}_{\Q}$ with the understanding that the implicit base is always $\Q$).

\begin{theorem}
\label{thm:main}
There exists $\delta > 0$ such that
\[
N^{\mathrm{proj}}(B)
=
c_{\LRS} \frac{B}{(\log B)^{n - \sum_{i = 1}^n \vert V_{\{i\}} \vert^{-1}}} \left(1 + O\left(\frac{1}{(\log B)^{\delta}}\right)\right),
\]
where the sets $V_{\{i\}}$ are described in Subsection~\ref{sec:indicator}. The constant $c_{\LRS}$ in this formula, given explicitly in Corollary~\ref{cLRS}, is precisely the leading constant predicted in~\cite[Conjecture~3.16]{LRS} and the exponent of $\log B$ is precisely the one predicted in $\cite{LS16}$ (c.f.~Lemma~\ref{lem:firstconj}). 
\end{theorem}

\begin{remark}
An equivalent way to frame our result is that~\eqref{eq:def} defines a product of conic bundles over $\PPP^{n - 1}$, so Theorem~\ref{thm:main} counts how often the fibres have a $\Q$-point. In fact, our results are a consequence of a more general result where the base is allowed to be $\AAA^n$ and the monomials are not imposed to have the same degree (c.f.~Theorem~\ref{tTechnicalHeart}). The condition of squarefreeness and coprimality on the $M_{i, j}$ can also be lifted, as one may always reduce to this case by multiplying variables and absorbing squares if needed. However, it is essential for our method that the $M_{i, j}$ are monomials.
%
\end{remark}

Loughran--Smeets~\cite{LS16} have extended the arguments which provided Serre's upper bounds to the setting of local solubility in more general fibrations. In particular, suppose that $X$ is a smooth projective variety admitting a dominant map $f: X \rightarrow \PPP^{n-1}$ with geometrically integral generic fibre. Define the counting function for everywhere locally soluble fibres to be
\[
N(f, B) \coloneqq \# \{ t \in \PPP^{n-1}(\Q) : H(t) \leq B,\ f^{-1}(t)(\AAA_\Q) \neq \varnothing \},
\] 
with $H$ the anticanonical height function on $ \PPP^{n-1}$. Loughran--Smeets showed that there exists an exponent $\Delta(f) \geq 0$ such that 
$$
N(f, B) \ll \frac{B}{(\log B)^{\Delta(f)}}.
$$ 
Moreover, if there is at least one everywhere locally soluble fibre and if every fibre over a codimension one point has at least one irreducible component of multiplicity 1, they predicted that their upper bound was of the correct order of magnitude for the counting function. This prediction, now known as the \emph{Loughran--Smeets conjecture}, has received a great deal of interest in recent years. The conjecture was extended to allow for smooth projective weak Fano bases (satisfying a mild technical assumption) by Loughran--Rome--Sofos~\cite{LRS}, who predicted the existence of an asymptotic formula and gave an explicit description of the expected leading constant. So far the conjectured lower bound has been proven in the following settings:
\begin{itemize}
\item $\Delta(f) =0$ over projective space or a quadric hypersurface (Loughran--Smeets~\cite{LS16} and Browning--Heath-Brown~\cite{BHB21});
\item Planar conics (Guo~\cite{Guo95}, Hooley~\cite{Hoo93,Hoo07}, Loughran--Rome--Sofos~\cite{LRS} and Destagnol--Lyczak--Sofos~\cite{DLS});
\item Brauer--Severi varieties over toric varieties (Loughran~\cite{Lou18});
\item Brauer--Severi varieties over semi-simple algebraic groups (Loughran--Takloo-Bighash--Tanimoto~\cite{LTBT20});
\item Pencils with all split fibres lying over a rational point (Loughran--Matthiesen~\cite{LM});
\item Conic bundles over hypersurfaces (Sofos--Visse~\cite{MR4353917} and Destagnol--Sofos~\cite{DS});
\item Generalised Fermat equations (Koymans--Paterson--Santens--Shute~\cite{KPSS}).
\end{itemize}
The scope of the conjecture has been expanded to fibrations with multiple fibres by Browning--Lyczak--Smeets~\cite{BLS}. In Browning--Lyczak--Sarapin~\cite{arXiv:2203.06881}, it was shown that the conjecture cannot na{\"i}vely be generalised to more general bases as they demonstrated a fibration over a quadric surface with more than $B/(\log B)^{\Delta(f)}$ everywhere locally soluble fibres. An asymptotic formula for their example was recently provided by Wilson~\cite{cameron}, which features an as yet unexplainable $\log\log B$ term and a peculiar constant which can be enlightened by the viewpoint we put forth in the present paper (c.f.~Subsection~\ref{ssSub} and Proposition~\ref{lTama}). In the original setting where the base of the fibration is projective space, all known asymptotic formulae conform to the prediction of Loughran--Rome--Sofos and our main result is no exception. 

\begin{remark}
The Loughran--Smeets problem asks only for fibres which are soluble in all completions of $\Q$, a priori a simpler task than considering global solubility of fibres. However if the fibres all satisfy the Hasse principle, then the problems are equivalent. This is the case in the setting of Theorem~\ref{thm:main} since all the fibres are products of conics to which the Hasse--Minkowski theorem applies.
\end{remark}

Serre's original problem was not in fact stated in terms of the solubility of equations in families but rather the zero loci of elements of the Brauer group. Given a subset $\sB \subseteq \Br \PPP^{n - 1}$, Serre was concerned with how often one can choose a point $t \in \PPP^{n - 1}(\Q)$ so that the Brauer classes in $\sB$ all become trivial in $\Br \Q$ when evaluated at $t$. Denote by 
\[
\PPP^{n - 1}(\Q)_{\sB} = \{ t \in \PPP^{n - 1}(\Q): b(t) = 0 \text{ for all } b \in\sB\},
\]
the zero locus of $\sB$. Given the correspondence between order 2 elements in the Brauer group and conics, we have the following re-interpretation of our main result. Write $\Br_{\Sub}(\PPP^{n-1},\sB)$ for the subordinate Brauer group (see Definition~\ref{dSub}). 

\begin{corollary}
\label{cMain}
Let $\sB$ be a finite subset of $\Br(\PPP^{n-1}/\{t_1 \cdots t_n = 0\})$ consisting of quaternion algebras of the form $(\pm M_1(\mathbf t),\pm M_2(\mathbf{t}))$ for $M_1$ and $M_2$ monic polynomials. Define 
$$
N(\sB, B) \coloneqq \#\{ t \in \PPP^{n - 1}(\Q)_{\mathscr B} : H(t) \leq B\}.
$$ 
Then for every class $\alpha \in \Br_{\Sub}(\PPP^{n - 1}, \sB)/\Br \Q$, there exist local factors $\sigma_v(\alpha)$ at each place $v$ of $\Q$ with
\[
N(\sB, B)
\sim
\frac{n^{n - 1 - \sum_{i = 1}^n \vert \partial_i( \langle \sB \rangle) \vert^{-1}}}{\prod_{i = 1}^n \Gamma \left(\vert \partial_i( \langle \sB \rangle) \vert^{-1} \right)}
\frac{B}{(\log B)^{n-\sum_{i = 1}^n \vert \partial_i( \langle \sB \rangle) \vert^{-1}}}
\sum_{\alpha \in \frac{\Br_{\Sub}(\PPP^{n - 1}, \mathscr B)}{\Br \Q}} \prod_v \sigma_v(f_\alpha),
\]
where $\partial_i$ denotes the residue map along the divisor $\{t_i = 0\} \subseteq \PPP^{n - 1}$. Explicit expressions for the local factors $\sigma_v(f_\alpha)$ can be found in Proposition~\ref{lTama} for $v$ an odd prime, in Lemma~\ref{lem:2inf} otherwise and the map $f_\alpha$ is described in Lemma~\ref{t:identification}.
\end{corollary}

We note that in all previous results of this type either $\mathscr{B}$ was given by a single element or by several algebraic Brauer classes. This is the first result of its kind handling an arbitrarily large collection of transcendental Brauer classes.


\subsection{Distribution of R\'edei symbols}
Let $a, b, c$ be three non-zero integers and write $\Delta(a), \Delta(b), \Delta(c)$ for the discriminants of respectively $\Q(\sqrt{a}), \Q(\sqrt{b}), \Q(\sqrt{c})$. Under the assumptions
\begin{equation}
\label{eRedeiDefined}
(a, b)_v = (a, c)_v = (b, c)_v = +1 \quad \text{ for all places } v, \quad \gcd(\Delta(a), \Delta(b), \Delta(c)) = 1,
\end{equation}
R\'edei~\cite{Redei} defined already in the 1930s the so-called R\'edei symbol $[a, b, c] \in \{\pm 1\}$, see also Stevenhagen's work~\cite{Stevenhagen} for a modern treatment. This symbol measures the splitting of $c$ in a certain dihedral extension of $\Q$ containing $\Q(\sqrt{a}, \sqrt{b})$. R\'edei proved that his symbol is trilinear and satisfies a surprising reciprocity law, namely
$$
[a, b, c] = [b, a, c] = [a, c, b].
$$
The R\'edei symbol has taken a prominent role in two remarkably different directions. Firstly, it can be used to compute the $8$-rank of the class group of a quadratic field as was already known to R\'edei. This theory was fruitfully applied by Smith~\cite{Smith} to compute the distribution of $8$-ranks of class groups. Secondly, there is a deep connection between the R\'edei symbol and the triple linking number in knot theory. This connection is explored in various works, see for example Morishita's book~\cite{Mori}. The R\'edei symbol has been generalised to arbitrary Galois modules in Koymans--Smith~\cite{KS}, which was then used to calculate the distribution of $3$-Selmer groups in cubic twist families.

Lenstra asked the following two questions regarding the statistics of the R\'edei symbol:
\begin{itemize}
\item How frequently is the R\'edei symbol defined, i.e.~how many triples $(a, b, c)$ satisfy the conditions in~\eqref{eRedeiDefined}?
\item How frequently is the R\'edei symbol equal to $+1$?
\end{itemize}
We shall answer both questions in this paper. The techniques in our paper are able to handle the condition~\eqref{eRedeiDefined} directly. For the second question, we combine the results in our paper with a recent ``trilinear large sieve'' for R\'edei symbols, see~\cite[Theorem~1.10]{KS}.

To state our result, we introduce the set
\begin{equation}\label{eq:redeisetdef}
S(B) \coloneqq \left\{(a, b, c) \in \Z^3 : 
\begin{array}{l}
1 \leq |a|, |b|, |c| \leq B, \\ 
a, b, c \text{ squarefree}, \\
a \equiv 1 \bmod 8 \text{ or } b \equiv 1 \bmod 8 \text{ or } c \equiv 1 \bmod 8, \\
(a, b)_v = (a, c)_v = (b, c)_v = 1\text{ for all places }v, \\ 
\gcd(\Delta(a), \Delta(b), \Delta(c)) = 1
\end{array}
\right\}.
\end{equation}
The restriction that at least one of $a, b, c$ is $1$ modulo $8$ is used to apply the trilinear large sieve from~\cite{KS}. By generalising the argument of~\cite{KS}, it should be straightforward to extend our results to arbitrary $a, b, c$. Given $\delta \in \{\pm 1\}$, we are concerned with the set
$$
S(B, \delta) = \{(a, b, c) \in S(B) : [a, b, c] = \delta\}.
$$

\begin{theorem}
\label{tRedei}
Let $\delta \in \{\pm 1\}$. Then, as $B \rightarrow \infty$, we have the asymptotic formula
$$
\lvert S(B, \delta)\rvert \sim \frac{49}{2} \times
\frac{B^3}{(\log B)^{\frac{9}{4}} \times 2^{\frac{7}{4}} \times\Gamma(1/4)^3} \times 
\prod_{p>2} \left( 1- \frac{1}{p}\right)^{\frac{3}{4}}\left( 1 + \frac{3}{4p}+\frac{ 3}{4 p^2} \right).
$$
\end{theorem}

\subsection{The role of the Brauer group in the character sum method}
\label{ssSub}
The proof of Theorem~\ref{thm:main} proceeds via the \emph{character sum method}, an approach to counting problems in arithmetic statistics particularly suited to counting solubility of conics, norm forms and Fano varieties. Although there is no explicit description of this method in the literature nor a direct link between works using the method, it is an approach to arithmetic statistics problems that is well-known to experts and has been applied to number theoretic problems as diverse as counting number fields, studying ranks in quadratic twists of elliptic curves~\cite{HB, HB2} and ranks of class groups~\cite{FK4rank}. These methods have inspired a great amount of recent results~\cite{AK, BC, Chan, FK, FKP, Klys1, Klys2, KMS, KR, MP, Rome, Santens, Smi22b}.

We attempt an informal description of what we refer to as the character sum method now so that we may explain the role the Brauer group plays in counting problems of Loughran--Smeets type and how this fits more generally into the character sum method.

The method is best applied to counting problems concerning tuples of integers with some local condition, which can be represented by characters, imposed at every prime. A prototypical example would be counting how frequently a Hilbert symbol (or collection thereof) takes the value +1. The broad strokes of the method in its modern form are derived from the papers of Heath-Brown~\cite{HB, HB2}, Fouvry--Kl{\"u}ners~\cite{FK4rank} and Friedlander--Iwaniec~\cite{FI}. In the simple example of checking how frequently the symbol $(a, b)_p$ is equal to +1 for all primes $p$ as $a$ and $b$ vary we can characterise the local condition with Legendre symbols
\begin{equation}
\label{eDetect}
\mathbf{1}_{(a, b)_p = +1} = \frac{1}{2} \left(1 + \left(\frac{b}{p}\right)\right).
\end{equation}
if $p\mid a$ and $p\nmid b$, and a similar expression works if $p\nmid a$ and $p\mid b$. Note that the case $p\mid a$ and $p\mid b$ happens with probability $O(p^{-2})$, and is therefore easy to treat separately.

This turns the problem into the evaluation of a complicated sum weighted by a product containing characters like~\eqref{eDetect} over all the prime divisors of the summation variables. This product can be expanded into a sum over divisors. After these manipulations, the summand is now a product of characters multiplied by an arithmetic weight $c^{-\omega(n)}$. This product of characters is usually non-principal and thus exhibits oscillation. In practice, this oscillation is rigorously demonstrated either via the Siegel--Walfisz theorem, if the variables involved are small, or otherwise using the large sieve.

However it is possible that a certain conspiracy between the variables appearing in the various characters arises causing the characters to cancel and thus producing a main term. Finding how these main terms arise is a combinatorial process referred to as linked indices. The set-up in the present paper necessitates a very general version of this process, carried out in Section~\ref{sComb}. Each of the main terms which arise have the same order of magnitude but possibly different leading constants. 

This peculiarity is explained by a deep connection with the Brauer group of the underlying variety. There is an implicit connection between the process just described and the influence the Brauer group has on the existence of rational points. Judging whether a $\Q_p$-point trivialises a quaternion algebra in the Brauer group is precisely equivalent to evaluating a Hilbert symbol. Performing this calculation at all primes allows one to compute the influence of the Brauer--Manin obstruction on the existence of rational points which conjecturally completely controls whether or not a fibre is soluble. This connection between main terms in such counts and Brauer group elements is well-known in the study of rational points (in particular when counting points in conic bundles such as in \cite{chats} or \cite{sumsofsquares}). 

The novelty in our approach is in Lemma~\ref{t:identification} where a direct explicit correspondence is produced between elements of the subordinate Brauer group and combinatorial data describing linked indices. This makes completely explicit the connection between the geometric and analytic sides of the problem and allows our leading constants to be written as a sum over terms indexed by the subordinate Brauer group, as in Corollary~\ref{cMain}. We note that this observation could be viewed as a more deeply hidden and obscure version of the phenomenon that when one counts soluble fibres over a toric variety (as in~\cite{Lou18}). Namely, after Poisson summation one is required to sum the contribution from non-trivial characters. This precisely corresponds to elements of the subordinate Brauer group.

It is our philosophy that all counting problems of this type should be understood in this way (main terms from the combinatorics correspond to Brauer group elements) and that this explicit connection enlightens previous results in the literature (in particular, the pair of main terms with different constants arising in Wilson~\cite{cameron} and Pagano--Sofos~\cite{PS}). While this connection is better understood in the realm of rational points, the philosophy should be applicable in all the domains of arithmetic statistics where the character sum method is used. With the recent advent of Brauer groups in field counting problems \cite{LSBG}, it seems likely that this perspective could have deep consequences throughout arithmetic statistics.

\subsection*{Acknowledgements}
This project was initiated when PK visited the ISTA, and all authors wish to thank the institute for its excellent working conditions. The authors would also like to thank Dan Loughran and Tim Santens for many inspiring conversations about subordinate Brauer groups.

PK gratefully acknowledges the support of Dr.~Max R\"ossler, the Walter Haefner Foundation and the ETH Z\"urich Foundation. PK also acknowledges the support of the Dutch Research Council (NWO) through the Veni grant ``New methods in arithmetic statistics''. NR was supported by FWF project ESP 441-NBL.

\section{The character sum}
\label{sChar}
In this section we shall rewrite our counting problem as a character sum, and then isolate the potential main terms. The arguments in this section are fairly routine. The key innovation in this paper is the treatment of the main terms in Sections~\ref{sComb} and~\ref{sLeading}, which requires more subtle and novel arguments.

\subsection{Changing variables}
Let $n \in \Z_{\geq 2}$, and consider the affine variety $V$ defined as the intersection of 
\begin{equation}
\label{eDefV}
\begin{alignedat}{4}
&M_{1, 1}(t_1, \dots, t_n) x_1^2 \, &&+ \, M_{1, 2}(t_1, \dots, t_n) y_1^2 \, && \, + M_{1, 3}(t_1, \dots, t_n) z_1^2 &&= 0 \\
&M_{2, 1}(t_1, \dots, t_n) x_2^2 \, &&+ \, M_{2, 2}(t_1, \dots, t_n) y_2^2 \, && \, + M_{2, 3}(t_1, \dots, t_n) z_2^2 &&= 0\\
& \quad \quad \quad \quad \ \vdots && \quad \quad \quad \quad \quad \ \ \vdots && \quad \quad \quad \quad \quad \ \ \vdots && \ \ \vdots \\
&M_{m, 1}(t_1, \dots, t_n) x_m^2 \, &&+ \, M_{m, 2}(t_1, \dots, t_n) y_m^2 && \, + M_{m, 3}(t_1, \dots, t_n) z_m^2 &&= 0,
\end{alignedat}
\end{equation}
where the $M_{i, j}$ are squarefree monomials with leading coefficient $\pm 1$. We assume that $M_{i, 1}$, $M_{i, 2}$ and $M_{i, 3}$ are pairwise coprime.

There is a natural morphism $\varphi:V \rightarrow \AAA^n$ such that each non-degenerate fibre is the direct product of $m$ conics. We will count how often this equation is locally soluble as we vary over tuples $(t_1, \dots, t_n)$ with $\gcd(t_1, \dots, t_n) = 1$ and $\max(|t_1|, \dots, |t_n|) \leq B$. Denote this count by $N(B)$. Since each non-degenerate fibre is a direct product of conics, this also gives results on the count of fibres that are globally soluble.

We let $\mathbf{s} = (s_i)_{1 \leq i \leq n}$ with each $s_i \in \{\pm 1\}$ be a collection of sign conditions for the $t_i$. We let $\boldsymbol{\lambda} = (\lambda_i)_{1 \leq i \leq n}$ with each $\lambda_i \in \{0,1\}$ be the parity of the $2$-adic valuation $\nu_2(t_i)$ of $t_i$. We let $\mathbf{tw} = (\tw_i)_{1 \leq i \leq n}$ be a set of congruence conditions modulo $8\cdot 2^{\lambda_i}$ for the squarefree part of $t_i$ (which is by definition $t_i$ divided by its largest square divisor). Let $\mathbf{u}=(u_i)_{1 \leq i \leq n}$ with each $u_i \in \{\pm1,\pm3\}$ so that $\tw_i = 2^{\lambda_i}u_i$. 

\begin{definition}
\label{def:admissible}
We say that the vector $\mathbf{s} \in \{\pm 1\}^n$ is admissible if $\varphi^{-1}( \mathbf{s})$ has an $\R$-point. We say that $(\mathbf u,\boldsymbol \lambda)$ is admissible if $\boldsymbol \lambda \in \{0,1\}^n - \{ \mathbf{1}\}$ and $\mathbf{u} \in \{\pm 1, \pm 3\}^n$ are such that $\varphi^{-1}(2^{\boldsymbol{\lambda}} \mathbf{u})$ has a $\Q_2$-point.
\end{definition}

We uniquely write 
$$
t_i = s_i 2^{\lambda_i} a_i b_i^2
$$ 
with $\lambda_i \in \{0, 1\}$, $a_i > 0$, $b_i > 0$ and $a_i$ squarefree and odd. We consider the subsum
\[
N(B, \mathbf{s}, \mathbf{u}, \boldsymbol{\lambda}) \coloneqq \sum_{\substack{t_1, \dots, t_n \in \Z \\ \gcd(t_1, \dots, t_n) = 1 \\ \max(|t_1|, \dots, |t_n|) \leq B \\\sgn(t_i) = s_i \\ \nu_2(t_i)\equiv \lambda_i\bmod 2\\ a_i \equiv s_iu_i \bmod 8}} \mathbf{1}_{\varphi^{-1}(\mathbf{t}) \text{ E.L.S.}}.
\]
We may assume that $\mathbf{s}$ and $(\mathbf{u}, \boldsymbol{\lambda})$ are admissible. Of course, we get back $N(B)$ by summing over the finitely many possibilities for the vectors $\mathbf{s}$, $\mathbf{u}$ and $\boldsymbol{\lambda}$. Our main counting result is in fact an asymptotic formula for $N(B, \mathbf{s}, \mathbf{u}, \boldsymbol{\lambda})$. As this requires a significant amount of notation to introduce that we will slowly develop over the course of the proof, we have opted to defer stating this result until Theorem~\ref{tTechnicalHeart}.

Denote $[n] \coloneqq \{1, \dots, n\}$. For every subset $\varnothing \subset S \subset [n]$, we introduce a new variable
\[
v_S \coloneqq \prod_{\substack{p \text{ prime} \\ p \mid a_i \Longleftrightarrow i \in S}} p.
\]
The new variables $v_S$ are positive, odd, squarefree and pairwise coprime. With this construction, we have
\[
a_i = \prod_{\substack{\varnothing \subset S \subset [n] \\ i \in S}} v_S.
\]
In particular, if $|S| \geq 2$ and $p \mid v_S$, then $p$ divides at least two of the $a_i$. 

\begin{lemma}
\label{lGCD}
The contribution to $N(B, \mathbf{s}, \mathbf{u}, \boldsymbol{\lambda})$ coming from $v_S > (\log B)^A$ for some $S$ with $|S| \geq 2$ is bounded by $O(B^n(\log B)^{-A})$, where the implied constant depends at most on $n$.
\end{lemma}

\begin{proof}
By relabelling we can assume that $v_S$ divides $a_1$ and $a_2$, so write $t_1' = t_1/v_S$ and $t_2' = t_2/v_S$. Observe that
\[
\sum_{v_S>(\log B)^A} \sum_{|t_1'| \leq B/v_S}\sum_{|t_2'|\leq B/v_S} 1\ll B^2(\log B)^{-A}.
\]
Summing this over $t_3, t_4, \dots$, and then over all possible $S$ with $|S| \geq 2$, we obtain the claim in the lemma.
\end{proof}

Letting $N_1(B, \mathbf{s}, \mathbf{u}, \boldsymbol{\lambda})$ be the subsum of $N(B, \mathbf{s}, \mathbf{u}, \boldsymbol{\lambda})$ with $v_S \leq (\log B)^{100}$ for $|S| \geq 2$, we have by Lemma~\ref{lGCD} the asymptotic formula
$$
N(B, \mathbf{s}, \mathbf{u}, \boldsymbol{\lambda}) = N_1(B, \mathbf{s}, \mathbf{u}, \boldsymbol{\lambda}) + O\left(\frac{B^n}{(\log B)^{100}}\right),
$$
where the implied constant depends only on $n$. Since $n$ is fixed for us, we shall not track the dependency on $n$ in all of our implied constants.

From now on we fix $\mathbf{b}$, introducing a new sum
\begin{equation}
\label{eq:N1def}
N_1(B, \mathbf{s}, \mathbf{u}, \boldsymbol{\lambda}, \mathbf{b}) \coloneqq \sum_{\substack{v_S > 0 \\ v_i \leq \frac{B}{2^{\lambda_i} \cdot b_i^2 \cdot \prod_{\{i\} \subset S} v_S} \\ \gcd(\{v_S\} \cup \{b_i : i \notin S\}) = 1 \\ \sgn(t_i) = s_i \\ \nu_2(t_i)\equiv \lambda_i\bmod 2\\ a_i \equiv s_iu_i \bmod 8\\ v_S \leq (\log B)^{100} \text{ for } |S| \geq 2}} \mu^2\left(2 \cdot \prod_S v_S\right) \times \mathbf{1}_{\varphi^{-1}(\mathbf{t}) \text{ E.L.S.}}
\end{equation}
leading to a decomposition of the type
\[
N_1(B, \mathbf{s}, \mathbf{u}, \boldsymbol{\lambda}) = \sum_{\mathbf{b}} N_1(B, \mathbf{s}, \mathbf{u}, \boldsymbol{\lambda}, \mathbf{b}).
\]

\begin{lemma}
The contribution to $N_1(B, \mathbf{s}, \mathbf{u}, \boldsymbol{\lambda})$ coming from vectors $\mathbf{b}$ with $b_i > (\log B)^A$ for some $i$ is bounded by $O(B^n(\log B)^{-A})$, where the implied constant depends only on $n$.
\end{lemma}

\begin{proof}
By relabelling, assume that $b_1>(\log B)^{A}$. Since $b_1^2\mid t_1$, write $t_1'=t_1/b_1^2$. Then the number of possible $t_1$ is bounded by 
\[
\sum_{b_1 > (\log B)^{A}} \sum_{|t_1'|\leq B/b_1^2} 1\ll B(\log B)^{-A}.
\]
Summing over the rest of the variables $t_2, \dots, t_n$ gives the bound.
\end{proof}

Henceforth we shall treat $\mathbf{b} = (b_i)_{1 \leq i \leq n}$ as a fixed number outside the main sum. We may and will assume that all the $b_i$ are small, say $b_i \leq (\log B)^{100}$.

\subsection{Expansion of the indicator function}
\label{sec:indicator}
Instead of thinking about monomials, we will think of them as subsets of $[n] \cup \{-\}$ in the natural way, where $[n] \coloneqq \{1, \dots, n\}$. Taking $t_{\{-\}} =s_{\{-\}}= -1$ by convention, this gives subsets $S_{i, j} \subseteq [n] \cup \{-\}$ for each $1 \leq i \leq m$ and $1 \leq j \leq 3$, corresponding to the monomial $M_{i, j}(t_1, \dots, t_n)$. Then we have in terms of the variables $v_S$
\[
M_{i, j}(t_1, \dots, t_n) \in \prod_{k \in S_{i, j}} \left(\sgn(t_k) \prod_{k\in S \subset [n]} v_S\right) \cdot (\Q^{\times})^2
= \prod_{k\in S_{i, j}} s_k \prod_{|S \cap S_{i, j}| \text{ odd}} v_S \cdot (\Q^{\times})^2.
\]
For each $\varnothing \subset S \subset [n]$, we specify a subspace $V_S$ on the power set of $[n] \cup \{-\}$ by taking the space generated by $\{g_{i, S} : 1 \leq i \leq m\}$, where $g_{i, S}$ is selected according to the table below. We denote the symmetric difference of sets by $+$.

\begin{table}[!ht]
\begin{center}
\begin{tabular}{|c||c|}
\hline
$(|S \cap S_{i, 1}|, |S \cap S_{i, 2}|, |S \cap S_{i, 3}|) \bmod 2$ & $g_{i, S}$ \\ \hline \hline
$(0, 0, 0)$,\quad $(1, 1, 1)$ & $\varnothing$ \\ \hline
$(1, 0, 0)$,\quad $(0, 1, 1)$ & $\{-\} + S_{i, 2}+ S_{i, 3}$ \\ \hline
$(0, 1, 0)$,\quad $(1, 0, 1)$ & $\{-\}+ S_{i, 1}+ S_{i, 3}$ \\ \hline
$(0, 0, 1)$,\quad $(1, 1, 0)$ & $\{-\}+ S_{i, 1} + S_{i, 2}$ \\ \hline
\end{tabular}
\end{center}
\end{table}

\vspace*{-\baselineskip}

If $\{-\}$ is in $V_S$, we define $W_S$ to be the unique splitting $V_S = \langle \{-\} \rangle \oplus W_S$, where all the sets in $W_S$ do not contain $\{-\}$. Otherwise, if $\{-\}$ is not in $V_S$, we define $W_S \coloneqq V_S$. We write $\mathcal{D}$ for the collection of $S$ for which $\{-\} \in V_S$. 

We now detect the everywhere locally soluble condition. The equations~\eqref{eDefV} are locally soluble at $2$ and $\infty$ by construction of $\mathbf{s}$, $ \mathbf{u}$ and $\boldsymbol{\lambda}$. Moreover, if $p$ is an odd finite prime not dividing $t_1 \cdots t_n$, then the equations~\eqref{eDefV} are also locally soluble. Indeed, this is a well-known fact for each individual conic, and since each fibre is a product of conics, the claim holds. The $g_{i,S}$ have been chosen precisely because the Legendre symbol $(\prod_{j \in g_{i,S}} t_j / p)$ encodes the solubility of the $i$-th conic at a prime $p$ dividing $v_S$. Therefore given admissible $\mathbf{s}$ and $(\mathbf{u},\boldsymbol{\lambda})$, we can substitute into~\eqref{eq:N1def} the expression
\[
\mathbf{1}_{\varphi^{-1}(\mathbf{t}) \text{ E.L.S.}}
=
\prod_S \prod_{p \mid v_S} \prod_{i = 1}^m \left(\frac{1}{2} \left(1 + \left(\frac{\prod_{j \in g_{i, S}} t_j}{p}\right)\right)\right),
\]
where $t_{\{-\}} = -1$. Note that in the product over $i$, some of the detector functions may be redundant. Indeed, we only need to check this over the vector space $W_S$, which we codify in our next lemma.

\begin{lemma}
\label{lExpand}
Let $\varnothing \subset S \subset [n]$. Then we have
\[
\frac{1}{2^m} \prod_{i = 1}^m \left(1 + \left(\frac{\prod_{j \in g_{i, S}} t_j}{p}\right)\right) = \frac{\mathbf{1}_{S \in \mathcal{D} \Rightarrow p \equiv 1 \bmod 4}}{|W_S|} \sum_{T \in W_S} \left(\frac{\prod_{j \in T} t_j}{p}\right).
\]
\end{lemma}

\begin{proof}
Pick a basis $\{h_1, \dots, h_k\}$ for $W_S$ over $\FF_2$, and extend this to a basis $\mathcal{H}$ for $V_S$ by adjoining the element $\{-\}$ if $\{-\} \in V_S$. Since both $\{g_{i, S} : i \in [m]\}$ and $\mathcal{H}$ are spanning sets of $V_S$, we conclude that 
$$
\forall i \in [m]: \left(\frac{\prod_{j \in g_{i, S}} t_j}{p}\right) = 1 \quad \quad \Longleftrightarrow \quad \quad \forall h \in \mathcal{H} : \left(\frac{\prod_{j \in h} t_j}{p}\right) = 1.
$$
Hence we can rewrite the expression on the left as
\begin{multline}
\label{eq:splitsum}
\frac{1}{2^m} \prod_{i = 1}^m \left(1 + \left(\frac{\prod_{j \in g_{i, S}} t_j}{p}\right)\right)
=
\frac{1}{|V_S|} \prod_{h\in \mathcal{H}} \left(1 + \left(\frac{\prod_{j \in h} t_j}{p}\right)\right)\\
=\left(\frac{1}{2} \left(1 + \left(\frac{\prod_{j \in \{-\}} t_j}{p}\right)\right)\right)^{\mathbf{1}_{\{-\}\in V_S}}\times\frac{1}{|W_S|} \prod_{i=1}^k \left(1 + \left(\frac{\prod_{j \in h_i} t_j}{p}\right)\right).
\end{multline}
If $\{-\}$ is in $V_S$, the first term becomes
\begin{equation}
\label{eq:minus}
\frac{1}{2} \left(1 + \left(\frac{\prod_{j \in \{-\}} t_j}{p}\right)\right)=
\frac{1}{2} \left(1 + \left(\frac{-1}{p}\right)\right)
= \mathbf{1}_{p \equiv 1 \bmod 4}.
\end{equation}
Expanding the product over $1 \leq i \leq k$, we get
\begin{equation}
\label{eq:wscont}
\prod_{i = 1}^k \left(1 + \left(\frac{\prod_{j \in h_i} t_j}{p}\right)\right) =
\sum_{(\delta_1, \dots, \delta_k) \in \FF_2^k} \left(\frac{\prod_{j \in \sum_{i=1}^k h_i^{\delta_i}} t_j}{p}\right)=\sum_{T \in W_S} \left(\frac{\prod_{j \in T} t_j}{p}\right),
\end{equation}
since the symmetric difference of sets $\sum_{i=1}^k h_i^{\delta_i}$ runs precisely over the elements of $W_S$ as $(\delta_1, \dots, \delta_k)$ runs over $\FF_2^k$. Putting~\eqref{eq:minus} and~\eqref{eq:wscont} into~\eqref{eq:splitsum} completes the proof.
\end{proof}

Lemma~\ref{lExpand} gives the equality
\[ 
\prod_{p \mid v_S} \prod_{i = 1}^m \left(\frac{1}{2} \left(1 + \left(\frac{\prod_{j \in g_{i, S}} t_j}{p}\right)\right)\right) = \prod_{p \mid v_S}\frac{\mathbf{1}_{S \in \mathcal{D} \Rightarrow p \equiv 1 \bmod 4}}{|W_S|} \sum_{T \in W_S} \left(\frac{\prod_{j \in T} t_j}{p}\right),
\]
which we will continue to expand. Note that
\[
\prod_{p \mid v_S}\sum_{T \in W_S} \left(\frac{\prod_{j \in T} t_j}{p}\right) = \sum_{(w_{S, T})_T} \prod_{T \in W_S} \left(\frac{\prod_{j \in T} t_j}{w_{S, T}}\right),
\]
where the sum over $(w_{S, T})_T$ ranges over all factorisations of the form $v_S = \prod_{T \in W_S} w_{S, T}$. We define $\mathcal{I}$ to be the set of pairs $(S, T)$, where $\varnothing \subset S \subset [n]$ and $T \in W_S$. By a further interchange of the sum and product
\[
\prod_S \prod_{p \mid v_S} \prod_{i = 1}^m \left(\frac{1}{2} \left(1 + \left(\frac{\prod_{j \in g_{i, S}} t_j}{p}\right)\right)\right) = \sum_{(w_{S, T})_{(S, T) \in \mathcal{I}}}\prod_S\prod_{T \in W_S} |W_S|^{-\omega(w_{S, T})} \left(\frac{\prod_{j \in T} t_j}{w_{S, T}}\right),
\]
where the sum over $(w_{S, T})_{(S, T) \in \mathcal{I}}$ ranges over all positive integers $w_{S, T}$ with $v_S = \prod_{T \in W_S} w_{S, T}$ for every $S$, and moreover $S \in \mathcal{D}$ and $p \mid w_{S, T}$ implies $p \equiv 1 \bmod 4$. Therefore, putting this back into the expression~\eqref{eq:N1def} for $N_1(B, \mathbf{s}, \mathbf{u}, \boldsymbol{\lambda}, \mathbf{b})$, we get
\begin{align}
\label{eq:N1Flat}
N_1(B, \mathbf{s}, \mathbf{u}, \boldsymbol{\lambda}, \mathbf{b}) = \sum_{(w_{S, T})}^\flat \mu^2\left(2 \cdot \prod_{S, T} w_{S, T}\right) \times \prod_{S, T} |W_S|^{-\omega(w_{S, T})} \times \prod_{S, T} \left(\frac{\prod_{j \in T} t_j}{w_{S, T}}\right),
\end{align}
where $\flat$ denotes the new summation conditions in the variables $w_{S, T}$ with $(S, T) \in \mathcal{I}$, namely
\begin{align*}
&w_{S,T}\text{ odd}, \hspace{2em} w_{S,T}\geq 1, &&S \in \mathcal{D} \text{ and } p \mid w_{S, T} \Rightarrow p \equiv 1 \bmod 4, \\ 
&\prod_{T \in W_{\{i\}}} w_{\{i\}, T} \leq \frac{B}{2^{\lambda_i} \cdot b_i^2 \cdot \prod_{S\supset \{i\}} \prod_{T \in W_S} w_{S, T}}, &&\gcd(\{w_{S, T}\} \cup \{b_i : i \notin S\}) = 1, \\ 
&\prod_{\substack{\varnothing \subset S \subset [n] \\ i \in S}} \prod_{T \in W_S} w_{S, T} \equiv s_iu_i \bmod 8, &&\prod_{T \in W_S} w_{S, T} \leq (\log B)^{100} \text{ for } |S| \geq 2.
\end{align*}
We now make the substitution
\begin{equation}
\label{eq:tsub}
t_j = s_j 2^{\lambda_j} b_j^2 \prod_{\substack{\varnothing \subset S \subset [n] \\ j \in S}} v_S = s_j 2^{\lambda_j} b_j^2 \prod_{\substack{\varnothing \subset S \subset [n] \\ j \in S}} \prod_{T \in W_S} w_{S, T}.
\end{equation}
Inserting~\eqref{eq:tsub} into equation~\eqref{eq:N1Flat}, we obtain
\[
N_1(B, \mathbf{s}, \mathbf{u}, \boldsymbol{\lambda}, \mathbf{b}) = \sum_{(w_{S, T})}^\flat F((w_{S, T})_{S, T}),
\]
where $F((w_{S, T})_{S, T})$ equals
\begin{equation}
\label{eDefF}
\frac{\mu^2\left(2 \cdot \prod_{S, T} w_{S, T}\right)}{\prod_{(S, T) \in \mathcal{I}} |W_S|^{\omega(w_{S, T})}} \times \prod_{(S, T) \in \mathcal{I}} \left(\frac{\prod_{j \in T} s_j2^{\lambda_j}}{w_{S, T}}\right) \times \prod_{(S_1, T_1) \in \mathcal{I}} \prod_{(S_2, T_2) \in \mathcal{I}} \left(\frac{w_{S_2, T_2}}{w_{S_1, T_1}}\right)^{|S_2 \cap T_1|}
\end{equation}
by definition, and where $s_{\{-\}} \coloneqq -1$ and $\lambda_{\{-\}} \coloneqq 0$ by convention.

\subsection{Linked indices and oscillation}
Our next goal is to study the Legendre symbols present in equation~\eqref{eDefF}. Of particular importance are Legendre symbols such that its flipped version also appears in equation~\eqref{eDefF}. This behavior is captured in our next definition.

\begin{definition}
\label{def:linked}
We say that $(S_1, T_1) \in \mathcal{I}$ and $(S_2, T_2) \in \mathcal{I}$ are linked if
\begin{equation}
\label{eLinked}
|S_2 \cap T_1| + |S_1 \cap T_2| \equiv 1 \bmod 2,
\end{equation}
and $|S_i| = 1$ holds for at least one of $i \in \{1,2\}$.
Otherwise, we say that they are unlinked. 
\end{definition}

\subsubsection{The large sieve}
Here we treat the case when $(\{i\}, T_1)$ and $(\{j\}, T_2)$ are linked, and $w_{\{i\}, T_1}$ and $w_{\{j\}, T_2}$ are not too small. Below is a result taken from~\cite[Lemma~15]{FK4rank}, which we will use.

\begin{lemma}[Large sieve]
\label{lemma:largesieve}
Let $\epsilon > 0$. Let $a_m$ and $b_n$ be complex numbers such that $|a_m|, |b_n| \leq 1$. Then for every $M, N \geq 1$, we have
\[
\sum_{m \leq M} \sum_{n \leq N} a_m b_n \mu^2(2m) \mu^2(2n) \leg{n}{m} \ll_{\epsilon} MN (M^{-\frac{1}{2} + \epsilon} + N^{-\frac{1}{2} + \epsilon}).
\]
\end{lemma}

\begin{lemma}
\label{lemma:LS}
Let $n$ and $\epsilon > 0$ be given. Then there exists $C > 0$ such that for all real numbers $A > 0$ and all linked indices $(\{i\}, T_1)$ and $(\{j\}, T_2)$
\[
\sum_{(w_{S, T})_{(S, T) \neq (\{i\}, T_1), (\{j\}, T_2)}}
\left|\sum_{w_{\{i\}, T_1},w_{\{j\}, T_2}>(\log B)^A}^\flat F((w_{S, T})_{S, T})\right|
\leq \frac{CB^n}{(\log B)^{(\frac{1}{2} - \epsilon) A - 100 \cdot 4^n}}.
\]
\end{lemma}

\begin{proof}
Freezing the values of $(w_{S, T})_{(S, T) \neq (\{i\}, T_1), (\{j\}, T_2)}$, we apply Lemma~\ref{lemma:largesieve} to
\[
\sum_{w_{\{i\}, T_1}} \sum_{w_{\{j\}, T_2}} F((w_{S, T})_{S, T}) \ll_\epsilon \frac{B}{\prod_{T_1' \neq T_1} w_{\{i\}, T_1'}} \cdot \frac{B}{\prod_{T_2' \neq T_2} w_{\{j\}, T_2'}} (\log B)^{-(\frac{1}{2}-\epsilon)A}.
\]
Summing over all $(w_{S, T})_{(S, T)\neq (\{i\}, T_1), (\{j\}, T_2)}$ gives the desired bound, upon noting that $|\mathcal{I}| \leq \sum_S |W_S| \leq (2^n)^2 = 4^n$ and $w_{S, T} \leq (\log B)^{100}$ for $|S| \geq 2$.
\end{proof}

\subsubsection{Siegel--Walfisz}
Next we treat the contribution when $w_{\{j\}, U} > \exp((\log B)^{1/4})$ and all indices $(S, T)\in\mathcal{I}$ linked to $(\{j\}, U)$ satisfy $w_{S, T} \leq (\log B)^A$ for a fixed real number $A > 0$. Our principal tool will be the Siegel--Walfisz theorem.

\begin{theorem}[Siegel--Walfisz]
\label{tSW}
Let $A > 1$ be a real number. Then there exists $C > 0$ such that for all $X \geq 100$, for all integers $q$ satisfying $q \leq (\log X)^A$ and all integers $d$ coprime to $q$
\[
\left|\sum_{\substack{p \leq X \\ p \equiv d \bmod q}} 1 - \frac{\Li(X)}{\phi(q)}\right| \leq \frac{CX}{(\log X)^A},
\]
where $\phi$ is Euler's totient function.
\end{theorem}

We shall not directly apply Theorem~\ref{tSW}. Instead we shall use it to verify the conditions~\eqref{eSWAssumption1} in Theorem~\ref{tKou}. This theorem gives uniform bounds on partial sums of $f$ provided that it satisfies a Siegel--Walfisz condition and is not too large on prime powers.

\begin{theorem}
\label{tKou}
Let $Q \geq 2$ be a parameter and let $f$ be a multiplicative function with
\begin{equation}
\label{eSWAssumption1}
\left|\sum_{p \leq Z} f(p) \log p\right| = O_A\left(\frac{Z}{(\log Z)^A}\right) \quad \quad \textup{ for } Z \geq Q
\end{equation}
for all $A > 0$. Let $k, J \in \Z_{\geq 1}$ and let $\epsilon > 0$ be a real number. Assume that $|f| \leq \tau_k$, where $\tau_k$ is the $k$-fold divisor function. Then we have the estimate
\[
\left|\sum_{n \leq X} f(n)\right| = O\left(\frac{X (\log Q)^{2k + J - 1}}{(\log X)^{J + 1}}\right)
\]
for all $X \geq e^{(\log Q)^{1 + \epsilon}}$. The implied constant depends only on $k$, $J$, $\epsilon$ and the implied constant in~\eqref{eSWAssumption1} for $A$ sufficiently large in terms of $k$, $J$, $\epsilon$ only.
\end{theorem}

\begin{proof}
This is a special case of~\cite[Theorem~13.2]{Kou} with $\kappa = 0$, see also \cite[Remark 13.3]{Kou}.
\end{proof}

Let $A > 0$ be a fixed real number. At this point, we are ready to bound
\[
\sum_{\substack{(w_{S, T})_{(S, T) \neq (\{j\}, U)} \\ w_{S, T} \leq (\log B)^A \text{ if } (S, T) \in \mathcal{U} \\ \prod_{(S, T) \in \mathcal{U}} w_{S, T} \neq 1}} \sum_{w_{\{j\}, U} > \exp((\log B)^{1/4})}^{\flat} F((w_{S, T})_{S, T}),
\]
where $\mathcal{U} = \{(S, T)\in \mathcal{I} :(S, T) \text{ and } (\{j\}, U) \text{ are linked}\}$. We apply the triangle inequality in order to isolate the variable $w_{\{j\}, U} > \exp((\log B)^{1/4})$ as follows
\begin{equation}
\label{eq:fsw1}
\sum_{\substack{(w_{S, T})_{(S, T) \neq (\{j\}, U)} \\ w_{S, T} \leq (\log B)^A \text{ if } (S, T) \in \mathcal{U}\\
\prod_{(S, T) \in \mathcal{U}} w_{S, T} \neq 1}} \left|\sum_{w_{\{j\}, U} > \exp((\log B)^{1/4})}^{\flat} f(w_{\{j\}, U})\right|,
\end{equation}
where $f(n)$ is the multiplicative function defined by
\begin{multline*}
\mu^2\left(2n \cdot \prod_{(S, T) \neq (\{j\}, U)} w_{S, T}\right) \times |W_{\{j\}}|^{-\omega(n)} \times \left(\frac{\prod_{k \in U} s_k 2^{\lambda_k}}{n}\right) \times \\ \times \prod_{S_2, T_2} \left(\frac{w_{S_2, T_2}}{n}\right)^{|S_2 \cap U|} \times \prod_{S_1, T_1} \left(\frac{n}{w_{S_1, T_1}}\right)^{|\{j\} \cap T_1|}.
\end{multline*}
Note that the summation condition $\flat$ fixes $w_{\{j\}, U}  \bmod 8$. To handle this condition, we multiply $f(n)$ by a sum of Dirichlet characters $\chi$ modulo $8$. Applying the triangle inequality again allows us to bound~\eqref{eq:fsw1} by
\[
\sum_\chi
\sum_{\substack{(w_{S, T})_{(S, T) \neq (\{j\}, U)} \\ w_{S, T} \leq (\log B)^A \text{ if } (S, T) \in \mathcal{U} \\
\prod_{(S, T) \in \mathcal{U}} w_{S, T} \neq 1}} \left|\sum_{w_{\{j\}, U}>\exp((\log B)^{1/4})}^{\flat} \chi f(w_{\{j\}, U})\right|,
\]
where $\chi$ runs over all Dirichlet characters modulo $8$.

We then apply Theorem~\ref{tKou} with $k \coloneqq 1$, $J$ a large integer in terms of $n$ only, $\epsilon \coloneqq 1/2$ and $Q \coloneqq \exp((\log B)^{1/8})$. The condition~\eqref{eSWAssumption1} of Theorem~\ref{tKou} is met thanks to Theorem~\ref{tSW}: here we claim that $\chi f$ has zero average over the primes by using that the variables are linked so the resulting Dirichlet character is not principal. 

Indeed, the condition $\prod_{(S, T) \in \mathcal{U}} w_{S, T} \neq 1$ implies that there exists $(S_1,T_1)$ linked to $(\{j\}, U)$ such that $p\mid w_{S_1, T_1}$ for some odd prime $p$. Since the $w_{S, T}$ are pairwise coprime, $p$ must not divide any $w_{S, T}$ other than $w_{S_1, T_1}$. Then the conductor of the associated Dirichlet character is divisible by $p$ because exactly one of $|S_1\cap U|$ and $ |\{j\} \cap T_1|$ is odd. This gives the claim.

Summing the resulting bound over the remaining variables shows that this case ends up in the error term.

\begin{lemma}
\label{lemma:SW}
Let $A > 0$ be a real number and let $n, J \in \Z_{\geq 1}$. Fix $(\{j\}, U)\in\mathcal{I}$ and let $\mathcal{U} = \{(S, T)\in\mathcal{I} : (S, T) \textup{ and } (\{j\}, U) \textup{ are linked}\}$. Then 
\[
\sum_{\substack{(w_{S, T})_{(S, T) \neq (\{j\}, U)} \\ w_{S, T} \leq (\log B)^A \textup{ if } (S, T) \in \mathcal{U} \\ \prod_{(S, T) \in \mathcal{U}} w_{S, T} \neq 1}} \left|\sum_{w_{\{j\}, U} > \exp((\log B)^{1/4})}^{\flat} F((w_{S, T})_{S, T})\right|
\ll 
\frac{B^n (\log B)^{\frac{J}{4} + 1}}{(\log B)^{J + 1 - 100 \cdot 4^n}},
\]
where the implied constant depends on $A, J, n$ only.
\end{lemma}

\section{Combinatorics}
\label{sComb}
We now study the expression $\prod_{S_1, T_1} \prod_{S_2, T_2} \left(\frac{w_{S_2, T_2}}{w_{S_1, T_1}}\right)^{|S_2 \cap T_1|}$ from a combinatorial perspective.

\subsection{Blocking sets}
We write $\mathcal{P}(A)$ for the power set of $A$, which we will view as a vector space over $\FF_2$ using symmetric difference of sets. Write $\pi: \mathcal{P}([n] \cup \{-\}) \rightarrow \mathcal{P}([n])$ for the natural projection map. Then we denote by $\OBMS$ the set of maps $f: \mathcal{P}([n]) \rightarrow \mathcal{P}([n] \cup \{-\})$ such that
\begin{itemize}
\item the image of singletons satisfy $f(\{i\}) \in V_{\{i\}}$ and $f(\varnothing) = \varnothing$,
\item the composed map $\pi\circ f: \mathcal{P}([n]) \rightarrow \mathcal{P}([n] \cup \{-\}) \rightarrow \mathcal{P}([n])$ is linear,
\item we have the alternating properties
$$
|\{i\} \cap f(\{j\})| = |\{j\} \cap f(\{i\})| , \quad |\{i\} \cap f(\{i\})| = 0.
$$
\end{itemize}
The space $\OBMS$ is naturally a vector space, where one can freely choose the minus sign for $f(S)$ for each $|S| \geq 2$. This space is an overarching space where all of our maps will naturally live. We will now define some natural quotients of $\OBMS$. 

\begin{definition}
Define $\FAlt$ to be the quotient of $\OBMS$, where we identify two maps $f, g \in \OBMS$ if $\pi(f(S)) = \pi(g(S))$ for all $|S| \geq 2$.

Define $\OAlt$ to be the quotient of $ \OBMS$, where we identify two maps $f, g \in \OBMS$ if $\pi(f(S)) = \pi(g(S))$ for all $|S| = 1$.

Define $\Alt$ to be the quotient of $ \OBMS$, where we identify two maps $f, g \in \OBMS$ if $\pi(f(S)) = \pi(g(S))$ for all $|S| \geq 1$ (equivalently $\pi \circ f = \pi \circ g$).
\end{definition}

We shall later see that $\FAlt$ is precisely the subordinate Brauer group for our problem. The relevance of $\Alt$ is that it shows up naturally in our character sum method, and is a quotient of $\FAlt$. Note that the quotient $\Alt$ effectively groups all maps that are the same up to the choice of minus signs. It turns out that it is convenient to fix a set of representatives for $\Alt$. There are several natural ways to do this, but it turns out that the one easiest for later calculations is the following diagram

\begin{equation}
\label{eq:Altses}
\begin{tikzcd}
0 \arrow{r} & \Jay \arrow{d}{=} \arrow{r} & \OBMS \arrow{r} \arrow{d}{q} & \OAlt\arrow{r} \arrow{d}{q} & 0 \\
0 \arrow{r} & \Jay \arrow{r} & \FAlt \arrow{r} & \Alt \arrow[swap]{ul}{s} \arrow{r} & 0,
\end{tikzcd}
\end{equation}
where $q$ denotes the natural quotient maps and where the set-theoretic section $s$ takes any representative in $ \OBMS$ such that the cardinality $|\{S : f(S) \in W_S\}|$ is maximised. We remark that this typically does not uniquely specify $s$. However, the composition $q \circ s$ is an homomorphism splitting the bottom exact sequence of~\eqref{eq:Altses}, and we define $\RAlt$ to be the image of $s$ in $ \OBMS$ (which we shall often view as sitting inside $\FAlt$ via the map $q$).

The vector space $\Jay$ is by definition the kernel, and can be explicitly thought of as maps $f: \mathcal{P}([n]) \rightarrow \mathcal{P}([n] \cup \{-\})$ with $f(\{i\}) \in \{\varnothing, \{-\}\}$ for $\{i\} \in \mathcal{D}$ and $f(S) = \varnothing$ otherwise, and is naturally isomorphic to the algebraic part of the subordinate Brauer group. Note that we may naturally identify $\Jay$ with an element of $\mathcal{P}(\{i : \{i\} \in \mathcal{D}\})$ (with the correspondence sending a map $f$ to the subset $J \coloneqq \{i : f(\{i\}) = \{-\}\}$), and we shall often do so.

We record some properties that will be useful later. Throughout this paper, we shall maintain the convention that the notation $\{i, j\}$ denotes a $2$-set, i.e.~$i$ and $j$ are implicitly assumed to be distinct.

\begin{lemma}
\label{lemma:projqs}
Suppose that $g \in \FAlt$ corresponds to $(J, f)$ where $J \in \Jay$ and $f \in \RAlt$ under the section $s$ given by~\eqref{eq:Altses}. Then we have for all $i, j, k \in [n]$ 
\[
|\{i\} \cap f(\{j\})| = |\{i\} \cap g(\{j\})| \quad \textup{ and } \quad
|\{-\} \cap f(\{k\})| + |\{k\}\cap J| = |\{-\} \cap g(\{k\})|.
\]
In particular, for all $S\subseteq [n]$, we have
\[
\sum_{k \in S} |\{-\} \cap f(\{k\})| + \sum_{\{i, j\} \subseteq S} |\{i\} \cap f(\{j\})| + |J\cap S| = \sum_{k \in S} |\{-\} \cap g(\{k\})| + \sum_{\{i,j\}\subseteq S}|\{i\} \cap g(\{j\})|.
\]
\end{lemma}

\begin{example}
At this stage we work out an elaborate example. The main point of the example is to disabuse the reader of the seductive but false idea that $f(\{i\}) \in V_{\{i\}}$ implies $f(S) \in V_S$ for $|S| \geq 2$. We consider the quaternion algebras
\[
\begin{cases}
(t_2t_3, t_1t_4), \\
(t_2t_5, t_1t_6), \\
(t_4t_5, t_3t_6).
\end{cases}
\]
Given these quaternion algebras, we start by computing the possibilities for the vector spaces $V_S$ when $S$ is a singleton.
\begin{center}
\begin{tabular}{|c||c|}
\hline
$i$ & $V_{\{i\}}$ \\ \hline\hline
$1$ & $\{\varnothing,\{2,3\},\{2,5\},\{3,5\}\}$ \\ \hline
$2$ & $\{\varnothing,\{1,4\},\{1,6\},\{4,6\}\}$ \\ \hline
$3$ & $\{\varnothing,\{1,4\},\{4,5\},\{1,5\}\}$ \\ \hline
$4$ & $\{\varnothing,\{2,3\},\{3,6\},\{2,6\}\}$ \\ \hline
$5$ & $\{\varnothing,\{1,6\},\{3,6\},\{1,3\}\}$ \\ \hline
$6$ & $\{\varnothing,\{2,5\},\{4,5\},\{2,4\}\}$ \\ \hline
\end{tabular}
\end{center}
The vector space $\FAlt \cong \Alt$ is generated by $f_1, f_2, f_3, f_4$ defined in the following table.
\begin{center}
\begin{tabular}{|c||c|c|c|c|}
\hline
$i$ & $f_1(\{i\})$ & $f_2(\{i\})$ & $f_3(\{i\})$ & $f_4(\{i\})$ \\ \hline\hline
$1$ & $\varnothing$ & $\{2,3\}$ & $\{3,5\}$ & $\varnothing$ \\ \hline
$2$ & $\varnothing$ & $\{1,4\}$ & $\varnothing$ & $\{4,6\}$ \\ \hline
$3$ & $\{4,5\}$ & $\{1,4\}$ & $\{1,5\}$ & $\varnothing$ \\ \hline
$4$ & $\{3,6\}$ & $\{2,3\}$ & $\varnothing$ & $\{2,6\}$ \\ \hline
$5$ & $\{3,6\}$ & $\varnothing$ & $\{1,3\}$ & $\varnothing$ \\ \hline
$6$ & $\{4,5\}$ & $\varnothing$ & $\varnothing$ & $\{2,4\}$ \\ \hline
\end{tabular}
\end{center}
Note that $\pi(f_3(\{1, 2\})) = \{3, 5\}$. This shows that $f_3(\{1, 2\}) $ is not in the subspace $V_{\{1, 2\}} = \{\varnothing,\ \{-, 1, 2, 3, 4\},\ \{-, 1, 2, 5, 6\},\ \{3, 4, 5, 6\}\}$.
\end{example}

We now introduce the central definition of this section. Recall the definition of unlinked indices from Definition~\ref{def:linked}. 

\begin{definition}
A blocking set is a subset $\mathcal{B} \subseteq \mathcal{I}$ such that
\begin{itemize}
\item for every $k \in [n]$, there exists some $T \in W_{\{k\}}$ such that $(\{k\}, T) \in \mathcal{B}$;
\item every $(\{k\}, T) \in \mathcal{B}$ (with $k \in [n]$ and $T \in W_{\{k\}}$) and every $(S, T') \in \mathcal{B}$ are unlinked.
\end{itemize}
\end{definition}

At this stage we are ready to completely classify the blocking sets. This is a critical ingredient in our arguments, and we shall see in Section~\ref{sLeading} how blocking sets are intrinsically related to the subordinate Brauer group. 

\begin{lemma}
\label{lBlock}
Let $\mathcal{B}$ be a blocking set. Then there exists a unique element $g \in \RAlt$ with
\begin{equation}
\label{eq:blocking}
\left\{(\{i\}, g(\{i\})) : i \in [n]\right\} \subseteq \mathcal{B} \subseteq \left\{(S, g(S)) : \varnothing \subset S \subset [n],\ g(S) \in W_S\right\}.
\end{equation}
Conversely, any set of this shape is a blocking set.
\end{lemma}

\begin{proof}
The converse is clear, so let us prove the first part. For the first part, let $\mathcal{B}$ be a blocking set. By the definition of blocking set, we may choose a subset $J_k \in W_{\{k\}}$ satisfying $(\{k\}, J_k) \in \mathcal{B}$. The elements $(\{i\}, J_i), (\{j\}, J_j)\in\mathcal{B}$ must be unlinked for all $i, j \in[n]$, so $|\{i\} \cap J_j| = |\{j\} \cap J_i|$. Since $W_{\{k\}}$ is generated by sets that do not contain $k$, we also have $|\{k\} \cap J_k| = 0$. Therefore there exists $f \in \OBMS$ such that $f(\{k\}) = J_k$ for every $k \in [n]$.

We claim that $\pi(f(S)) = T$ for all $(S, T) \in \mathcal{B}$. Since $(\{k\}, J_k)$ and $(S, T)$ are unlinked, we get from negating equation~\eqref{eLinked}
\[
k \in T \Longleftrightarrow |S \cap J_k| \equiv 1 \bmod 2.
\]
Since $f \in \OBMS$, we can readily check that $|S \cap f(\{k\})| \equiv |f(S) \cap \{k\}| \bmod 2$. Furthermore, by the construction of $f$ we have $f(\{k\}) = J_k$, so
\[
k \in T \Longleftrightarrow |f(S)\cap \{k\}| = 1,
\]
which proves the claim. 

Note that for a given $T$, there is at most one set $U \in W_S$ with $\pi(U) = T$ by construction of $W_S$. Hence, we deduce from the claim that there exists some $f' \in \OBMS$ with
$$
\left\{(\{i\}, f'(\{i\})) : i \in [n]\right\} \subseteq \mathcal{B} \subseteq \left\{(S, f'(S)) : \varnothing \subset S \subset [n],\ f'(S) \in W_S\right\}.
$$
There is a unique representative $g \in \RAlt$ satisfying the first inclusion in equation~\eqref{eq:blocking}. Since $g$ maximizes the cardinality $| \{S : f'(S) \in W_S\}|$ among all $f'$ lifting a given map $\mathcal{P}([n]) \rightarrow \mathcal{P}([n])$ by construction, the second inclusion in equation~\eqref{eq:blocking} holds.
\end{proof}

We split the sum according to whether $w_{\{k\}, T} > \exp((\log B)^{1/4}) \eqqcolon Q'$ for $k \in [n]$ and whether $w_{S, T} = 1$ for $|S| \geq 2$. For an arbitrary subset $\mathcal{J} \subseteq \mathcal{I}$, we define the restricted sum
\[
N_1(B, \mathbf{s}, \mathbf{u}, \boldsymbol{\lambda}, \mathbf{b}, \mathcal{J}) \coloneqq \sum_{\substack{(w_{S, T})_{|S| \geq 2} \\ 1 < w_{S, T} \leq (\log B)^{100} \text{ if } (S, T) \in \mathcal{J} \\ w_{S, T} = 1 \text{ if } (S, T) \notin \mathcal{J}}}\sum_{\substack{(w_{\{k\}, T})_k \\ w_{\{k\}, T} > Q' \text{ if } (\{k\}, T) \in \mathcal{J} \\ w_{\{k\}, T} \leq Q' \text{ if } (\{k\}, T) \notin \mathcal{J}}}^{\flat} F((w_{S, T})_{S, T}).
\]
It is clear that
$$
N_1(B, \mathbf{s}, \mathbf{u}, \boldsymbol{\lambda}, \mathbf{b}) = \sum_{\mathcal{J} \subseteq \mathcal{I}} N_1(B, \mathbf{s}, \mathbf{u}, \boldsymbol{\lambda}, \mathbf{b}, \mathcal{J}).
$$
The main contribution in this decomposition will come from those $\mathcal{J}$ that are blocking sets. 

\begin{lemma}
\label{lemma:notblocking}
There exists $\delta' > 0$ depending only on $n$ such that
\[
\sum_{\mathcal{J} \textup{ is not a blocking set}} N_1(B, \mathbf{s}, \mathbf{u}, \boldsymbol{\lambda}, \mathbf{b}, \mathcal{J}) \ll \frac{B^n}{(\log B)^{\delta'}}.
\]
\end{lemma}

\begin{proof}
The total number of possible $\mathcal{J}$ depends only on $n$, so it suffices to bound each $N_1(B, \mathbf{s}, \mathbf{u}, \boldsymbol{\lambda}, \mathbf{b}, \mathcal{J})$ individually. The set $\mathcal{J}$ imposes the summation conditions
\[ 
\begin{array}{cl}
w_{\{k\}, T} > Q' & \text{for } (\{k\}, T) \in \mathcal{J} \\
w_{\{k\}, T} \leq Q' & \text{for } (\{k\}, T) \notin \mathcal{J} \\
1 < w_{S, T} \leq (\log B)^{100} & \text{for } (S, T) \in \mathcal{J}, |S| \geq 2 \\
w_{S, T} = 1 & \text{for } (S, T) \notin \mathcal{J}, |S| \geq 2.
\end{array}
\]
We now distinguish various cases depending on the structure of $\mathcal{J}$. Firstly, suppose that there exists some $k \in [n]$ such that $(\{k\}, T) \not \in \mathcal{J}$ for all $T$. In this case, the sum is trivially bounded by $\ll_\epsilon B^{n - 1 + \epsilon}$. Henceforth, we may assume that for every $k \in [n]$, there exists some $T_k$ such that $(\{k\}, T_k) \in \mathcal{J}$. 

Secondly, suppose that $(\{k\}, T) \in \mathcal{J}$ is linked to some $(S, T') \in \mathcal{J}$ with $|S| = 1$. In this case, we bound the contribution from such $\mathcal{J}$ by Lemma~\ref{lemma:LS}. 

Thirdly, if some $(\{k\}, T) \in \mathcal{J}$ is linked to some $(S, T') \in \mathcal{J}$ with $|S| \geq 2$, then we proceed as follows. For all $(\{h\}, T'')$ linked to $(\{k\}, T) \in \mathcal{J}$, we must have $(\{h\}, T'') \not \in \mathcal{J}$ by the previous case, and we can bound the contribution of the range $(\log B)^A < w_{\{h\}, T''} \leq Q'$ to $\mathcal{J}$ by another appeal to Lemma~\ref{lemma:LS}, where we take $A$ to be a large real number in terms of $n$ only. In the remaining range, we can apply Lemma~\ref{lemma:SW}. 

Therefore, for the remaining $\mathcal{J}$, all $(\{k\}, T) \in \mathcal{J}$ and $(S, T') \in \mathcal{J}$ are unlinked. Hence $\mathcal{J}$ is a blocking set, as desired.
\end{proof}
 
Now, in light of Lemma~\ref{lemma:notblocking}, assume that $\mathcal{J} = \mathcal{B}$ is a blocking set. By Lemma~\ref{lBlock}, any element $(\{k'\}, T') \notin \mathcal{B}$ is linked to some $(\{k\}, T) \in \mathcal{B}$. Therefore we have the estimate
\[
\sum_{\substack{(w_{S, T})_{|S| \geq 2} \\ 1<w_{S, T} \leq (\log B)^{100} \text{ if } (S, T) \in \mathcal{B} \\ w_{S, T} = 1 \text{ if } (S, T)\notin\mathcal{B}}} \sum_{\substack{(w_{\{k\}, T})_k \\ w_{\{k\}, T} > Q' \text{ if } (\{k\}, T)\in\mathcal{B}\\ 1 < w_{\{k\}, T} \leq Q' \text{ if } (\{k\}, T) \notin \mathcal{B}}}^{\flat} F((w_{S, T})_{S, T}) \ll B^n (\log B)^{-\delta'}
\]
for some $\delta' > 0$ by Lemma~\ref{lemma:LS} and Lemma~\ref{lemma:SW}. Thus we may assume that $w_{\{k'\}, T'} = 1$ for all $(\{k'\}, T') \notin \mathcal{B}$. Moreover, the condition $w_{\{k\}, T} > Q'$ for $(\{k\}, T) \in \mathcal{B}$ can now be dropped at the cost of an acceptable error term. Set
\[ 
N_2(B, \mathbf{s}, \mathbf{u}, \boldsymbol{\lambda}, \mathbf{b})\coloneqq \sum_{\mathcal{B}\text{ is a blocking set}} N_1'(B, \mathbf{s}, \mathbf{u}, \boldsymbol{\lambda}, \mathbf{b}, \mathcal{B}),
\]
where $N_1'(B, \mathbf{s}, \mathbf{u}, \boldsymbol{\lambda}, \mathbf{b}, \mathcal{B})$ imposes the additional condition $w_{\{k'\}, T'} = 1$ for all $(\{k'\}, T') \notin \mathcal{B}$ but does not impose that $w_{\{k\}, T} > Q'$ for $(\{k\}, T) \in \mathcal{B}$. Lemma~\ref{lBlock} allows us to associate a unique $f \in \RAlt$ to each $\mathcal{B}$. Given each $f \in \RAlt$, we identify $w_{S, f(S)}$ with a new variable $w_S$ and collect all blocking sets $\mathcal{B}$ allowed by~\eqref{eq:blocking}. Then the sum becomes
\begin{multline}
\label{eSumNow}
N_2(B, \mathbf{s}, \mathbf{u}, \boldsymbol{\lambda}, \mathbf{b}) =
\sum_{f \in \RAlt} \sum_{\substack{(w_S)_{S} \\ w_S = 1\text{ if }f(S)\notin W_S}}^{\flat}
\mu^2\left(2 \cdot \prod_{S} w_{S}\right) \times \\
\times\prod_{S} |W_S|^{-\omega(w_{S})} \times \prod_{S} \left(\frac{\prod_{j \in f(S)} s_j2^{\lambda_j} }{w_{S}}\right) \times \prod_{S_1} \prod_{S_2} \left(\frac{w_{S_2}}{w_{S_1}}\right)^{|S_2 \cap f(S_1)|},
\end{multline}
where the new summation conditions imposed by $\flat$ are
\begin{align*}
&w_S \text{ odd}, \hspace{2em} w_S \geq 1, &&S \in \mathcal{D} \text{ and } p \mid w_S \Rightarrow p \equiv 1 \bmod 4, \\ 
&w_{\{i\}} \leq \frac{B}{2^{\lambda_i} \cdot b_i^2 \cdot \prod_{S\supset \{i\}} w_S}, &&\gcd(\{w_S\} \cup \{b_i : i \notin S\}) = 1, \\ 
&\prod_{\substack{\varnothing \subset S \subset [n] \\ i \in S}} w_S \equiv s_iu_i \bmod 8, &&w_S \leq (\log B)^{100} \text{ for } |S| \geq 2.
\end{align*}

\subsection{Simplifying the sum}
At this stage, our main objective is to simplify the above sum~\eqref{eSumNow} as much as possible. By construction of $f$, we see that the Legendre symbol $(w_{S_2}/w_{S_1})$ appears if and only if $(w_{S_1}/w_{S_2})$ appears in the above sum. This will eventually lead to a rather explicit description of the sum, for each fixed $f$, after several combinatorial manipulations and applications of quadratic reciprocity. To start, we recall that $s_{\{-\}} = -1$ by convention and we observe that
\begin{align}
\label{eSimplifyLegendre1}
&\prod_{\varnothing \subset S \subset [n]} \left(\frac{\prod_{j \in f(S)} s_j2^{\lambda_j} }{w_{S}}\right) \times \prod_{\varnothing \subset S_1 \subset [n]} \prod_{\varnothing \subset S_2 \subset [n]} \left(\frac{w_{S_2}}{w_{S_1}}\right)^{|S_2 \cap f(S_1)|}\nonumber \\
&= \prod_{\varnothing \subset S \subset [n]} (-1,w_{S})_2^{|\{-\}\cap f(S)|} \times \prod_{\varnothing \subset S \subset [n]} \prod_{j \in [n]} \left(\left(\frac{s_j2^{\lambda_j} }{w_S}\right)^{|\{j\} \cap f(S)|} \times \prod_{\substack{S_2\neq S\\j\in S_2}} \left(\frac{w_{S_2}}{w_S}\right)^{|\{j\} \cap f(S)|}\right)\nonumber \\
&= \prod_{\varnothing \subset S \subset [n]} (-1,w_{S})_2^{|\{-\}\cap f(S)|} \times \prod_{\varnothing \subset S \subset [n]} \prod_{j \in [n]} \left(\frac{x_j/w_S^{|\{j\}\cap S|}}{w_S}\right)^{|\{j\} \cap f(S)|},
\end{align}
where we have defined $x_j \coloneqq s_j2^{\lambda_j} \prod_{\substack{\varnothing \subset S_2 \subset [n] \\ j \in S_2}} w_{S_2}$. We continue to manipulate the last product in~\eqref{eSimplifyLegendre1}
\begin{multline*}
\prod_{\varnothing \subset S \subset [n]} \prod_{j \in [n]} \left(\frac{x_j/w_S^{|\{j\}\cap S|}}{w_S}\right)^{|\{j\} \cap f(S)|} = \prod_{i \in [n]} \prod_{\substack{j \in [n] \\ j \neq i}} \prod_{\substack{\varnothing \subset S \subset [n] \\ i \in S}}
\left(\frac{x_j/w_S^{|\{j\} \cap S|}}{w_S}\right)^{|\{j\} \cap f(\{i\})|} \\
=
\prod_{\{i, j\} \subseteq [n]}
\left(\prod_{\substack{\varnothing \subset S \subset [n] \\ i \in S, \ j\notin S}}
\left(\frac{x_j}{w_S}\right)
\cdot \prod_{\substack{\varnothing \subset S \subset [n] \\ j \in S, \ i\notin S}}
\left(\frac{x_i}{w_S}\right)
\cdot \prod_{\substack{\varnothing \subset S \subset [n] \\ \{i, j\} \subseteq S}}
\left(\frac{x_ix_j/w_S^2}{w_S}\right)
\right)^{|\{j\} \cap f(\{i\})|},
\end{multline*}
where the last identity comes from grouping the terms $(i, j)$ with $(j, i)$ (note that the $i = j$ terms vanish because $\{i\} \cap f(\{i\}) = \varnothing$) and from considering the three cases for $S \cap \{i, j\} \neq \varnothing$. We apply Hilbert reciprocity to $(x_i, x_j)$ to rewrite this as
\begin{align*}
&=
\prod_{\{i, j\} \subseteq [n]}
\left((x_i, x_j)_2 (x_i, x_j)_{\infty}
\cdot \prod_{\substack{S\\ \{i, j\}\subseteq S}}\left(\frac{-x_i x_j/\gcd(x_i, x_j)^2}{w_S}\right)
\left(\frac{x_i x_j/w_S^2}{w_S}\right)
\right)^{|\{j\} \cap f(\{i\})|} \\
&= \prod_{\{i, j\} \subseteq [n]} (\tw_i, \tw_j)_2^{|\{j\} \cap f(\{i\})|} (s_i, s_j)_2^{|\{j\} \cap f(\{i\})|}
\cdot \prod_{|S| \geq 2} \prod_{\{i, j\} \subseteq S} (-1,w_S)_2^{|\{j\} \cap f(\{i\})|}.
\end{align*}
Noting that 
$$
\tw_i s_i \equiv x_i s_i \equiv 2^{\lambda_i} \prod_{\substack{\varnothing \subset S \subset [n] \\ i \in S}} w_S \bmod 8, 
$$ 
and $(-1, 2)_2 = +1$, the first term in~\eqref{eSimplifyLegendre1} equals
\begin{align*}
\prod_{\varnothing \subset S \subset [n]} \hspace{-0.1cm} (-1,w_S)_2^{|\{-\} \cap f(S)|} &= \prod_i (-1, \tw_i s_i \prod_{\substack{|S| \geq 2 \\ i \in S}} w_S)_2^{|\{-\} \cap f(\{i\})|} \cdot \prod_{|S| \geq 2} (-1, w_S)_2^{|\{-\}\cap f(S)|} \\
&= \prod_i (-1, \tw_is_i)_2^{|\{-\} \cap f(\{i\})|} \cdot \prod_{|S| \geq 2} (-1, w_S)_2^{|\{-\}\cap f(S)| + \sum_{i \in S} |\{-\} \cap f(\{i\})|}.
\end{align*}
Putting equation~\eqref{eSimplifyLegendre1} and the above two expressions back into the sum~\eqref{eSumNow}, we get the end result of this section. We formally state it now as this is the key intermediate result we shall build on in Section~\ref{sEuler}.

\begin{theorem}
\label{tIntermediate}
We have
\begin{multline}
\label{eFinalSum}
N_2(B, \mathbf{s}, \mathbf{u}, \boldsymbol{\lambda}, \mathbf{b}) = \sum_{f \in \RAlt} \prod_i (-1, \tw_is_i)_2^{|\{-\} \cap f(\{i\})|} \prod_{\{i, j\} \subseteq [n]} \left((\tw_i, \tw_j)_2 (s_i, s_j)_2\right)^{|\{j\} \cap f(\{i\})|} \\
\sum_{\substack{(w_S)_S \\ w_S = 1 \text{ if } f(S) \notin W_S}}^\flat \hspace{-0.5cm} \frac{\mu^2\left(2 \cdot \prod_S w_S\right)}{\prod_S |W_S|^{\omega(w_S)}} \prod_{\substack{S \\ |S| \geq 2}} (-1, w_S)_2^{|\{-\}\cap f(S)| + \sum_{i \in S} |\{-\} \cap f(\{i\})| + \sum_{\{i, j\} \subseteq S} |\{j\} \cap f(\{i\})|}.
\end{multline}
\end{theorem}

The major contribution to this sum will come from the singletons $S = \{i\}$ as for every set of greater size the variable $w_S$ appears in at least two height constraints and thus the sum over these variables will converge absolutely. We will asymptotically evaluate equation~\eqref{eFinalSum} resulting in a finite sum of Euler products, a feat that we shall accomplish in Section~\ref{sEuler}. In fact, we get $|\Jay|$ Euler products for each $f \in \RAlt$. Afterwards, we shall see how the exponent of $(-1, w_S)_2$ is related to the obstruction coming from the subordinate Brauer group. This will be done in Lemma~\ref{lemma:HScond1}. We will work out the conjectured leading constant from~\cite[Conjecture~3.16]{LRS} in Section~\ref{sLeading}.
 
\section{The Euler product}
\label{sEuler}
The goal of this section is to evaluate the sum from Theorem \ref{tIntermediate}. It turns out that each $f$ will give rise to $|\Jay|$ main terms, each with potentially distinct Euler products. 

\subsection{Analytic lemmata}
In this subsection we state some analytic preparations that will be useful in calculating the Euler product.

\begin{lemma}
\label{lem:LSD}
For $c \geq 1$, $\alpha \in \{1, 3,5,7\}$ and $d$ even, we have uniformly for $X \geq 2$ and $d \leq X$
\[
\sum_{\substack{n \leq X \\ \gcd(n, d)=1 \\ n \equiv \alpha \bmod 8}} \frac{1}{c^{\,\omega(n)}} = \beta_c \prod_{\substack{p \mid d \\ p>2}} \left( 1 + \frac{1}{c(p - 1)} \right)^{-1} \frac{X}{(\log X)^{1 - \frac{1}{c}}} + O \left( \frac{X}{(\log X)^{2 - \frac{1}{c}}} \right),
\]
where $\beta_c = \frac{1}{2^{2+\frac{1}{c}}}\frac{1}{\Gamma(1/c)} \prod_{p>2} \left( 1 + \frac{1}{c(p-1)} \right) \left( 1 - \frac{1}{p} \right)^{\frac{1}{c}}$.
\end{lemma}

\begin{proof}
This follows from~\cite[Theorem~13.2]{Kou}.
\end{proof}

\begin{lemma}
\label{lem:LSD2}
For $c \geq 1$, $\alpha \in \{1, 5\}$ and $d$ even, we have uniformly for $X \geq 2$ and $d \leq X$
\[
\sum_{\substack{n \leq X \\ \gcd(n, d) = 1 \\ n \equiv \alpha \bmod 8\\ p \mid n \Rightarrow p \equiv 1 \bmod 4}} \frac{1}{c^{\,\omega(n)}} = \widetilde{\beta}_c \prod_{\substack{p \mid d\\ p \equiv 1 \bmod 4}} \left( 1 + \frac{1}{c(p-1)} \right)^{-1} \frac{X}{(\log X)^{1 - \frac{1}{2c}}} + O \left( \frac{X}{(\log X)^{2 - \frac{1}{2c}}} \right),
\]
where $\widetilde{\beta}_c = \frac{1}{2^{1 + \frac{1}{2c}}} \frac{1}{\Gamma(1/2c)} \prod_{p > 2} \left( 1 + \frac{\mathbf{1}_{p \equiv 1 \bmod 4}}{c(p - 1)} \right) \left( 1 - \frac{1}{p} \right)^{\frac{1}{2c}}$.
\end{lemma}

\begin{proof}
This follows from~\cite[Theorem~13.2]{Kou}.
\end{proof}

\begin{lemma}
\label{lem:sofos}
For all $k \in \Z_{\geq 1}$ and $\epsilon > 0$, we have 
\[
\sum_{\substack{n \in \Z_{\geq 1} \\ k^2 \mid n, \ n \mid k^\infty}} \frac{1}{n^{\frac{1}{2} + \epsilon}} \ll_\epsilon \frac{1}{k^{1 + \epsilon}}.
\]
\end{lemma}

\begin{proof}
This is~\cite[Lemma~5.7]{LRS}.
\end{proof}

We will now combine the various analytic tools from this subsection to prove our main counting theorem.

\begin{theorem}
\label{theorem:maincounting}
Let $1 \leq r \leq n$ be integers. Let $c_1, \ldots, c_n \geq 1$ be real numbers and $d_1, \ldots, d_n$ be even natural numbers. Let $\alpha_i$ be odd integers with $\alpha_i \equiv 1\bmod 4$ for $i = 1, \ldots, r$. Let 
\[
\gamma_i\coloneqq
\begin{cases} 
 \hfil \frac{1}{2c_i} &\textup{if } 1 \leq i \leq r, \\
 \hfil \frac{1}{c_i} &\textup{if } r + 1 \leq i \leq n
\end{cases}
\]
and $\gamma \coloneqq \sum_{i = 1}^n \gamma_i$. Let 
\[
h_i(p) \coloneqq
\begin{cases} 
 \hfil 1 &\textup{if } p \equiv 1 \bmod 4, \\
 \hfil \mathbf{1}_{r + 1 \leq i \leq n} &\textup{if } p \equiv 3 \bmod 4.
\end{cases}
\]
We multiplicatively extend $h_i$ to all odd integers. Then
\[
\sum_{\substack{\mathbf{a} \in \Z_{\geq 1}^n, \ a_i \leq X_i \\ \gcd(a_i, d_i) = 1 \\ a_i \equiv \alpha_i \bmod 8}} 
\mu^2(a_1 \cdots a_n) \prod_{i = 1}^n \frac{ h_i(a_i) }{ c_i^{\omega(a_i)} } 
=
\frac{ \mathbf{F}(d_1, \dots, d_n)C X_1 \cdots X_n}{\prod_{i=1}^n (\log X_i)^{1-\gamma_i}} \left( 1 + O \left( (\log \min_i X_i)^{-1}\right) \right),
\]
where
\[
C = 
\frac{2^{-\gamma -(2n-r)}}{\prod_{i=1}^n\Gamma(\gamma_i)}
\prod_{p>2} \left( 1+ \frac{1}{p}\sum_{i=1}^n\frac{h_i(p)}{c_i}\right)
\left( 1 - \frac{1}{p} \right)^{\gamma},
\]
and
\begin{equation}
\label{eq:Fdef}
\mathbf{F}(d_1, \dots, d_n) = 
\prod_{\substack{p \mid d_1 \cdots d_n \\ p>2}}
\left(1 + \frac{1}{p} \sum_{i=1}^n \frac{h_i(p)}{c_i}\right)^{-1}
\left( 1+
\frac{1}{p}\sum_{\substack{i \in [n] \\ p \nmid d_i}}
\frac{h_i(p)}{c_i}
\right).
\end{equation}
\end{theorem}

\begin{proof}
We remove the squarefree condition with the standard identity for $\mu^2$ to get
\[
\sum_{\substack{\mathbf{a} \in \Z_{\geq 1}^n, \ a_i \leq X_i \\ \gcd(a_i, d_i) = 1 \\ a_i \equiv \alpha_i \bmod 8}} 
\mu^2(a_1 \cdots a_n) \prod_{i = 1}^n \frac{ h_i(a_i) }{ c_i^{\omega(a_i)} } 
=
\sum_{\ell \text{ odd}} \mu(\ell) 
\sum_{\substack{\mathbf{a} \in \Z_{\geq 1}^n, \ a_i \leq X_i \\ \gcd(a_i, d_i) = 1 \\ a_i \equiv \alpha_i \bmod 8 \\ \ell^2 \mid a_1 \cdots a_n}} 
\prod_{i=1}^n\frac{ h_i(a_i) }{ c_i^{\omega(a_i)} } .
\]
We write $a_i = a_i'a_i''$ such that $a_i' \mid \ell^\infty$ and $\gcd(a_i'', \ell) = 1$, so 
\begin{equation}
\label{eq:asum}
\sum_{\substack{\mathbf{a} \in \Z_{\geq 1}^n, \ a_i \leq X_i \\ \gcd(a_i, d_i) = 1 \\ a_i \equiv \alpha_i \bmod 8 \\ \ell^2 \mid a_1 \cdots a_n}} 
\prod_{i = 1}^n \frac{ h_i(a_i) }{ c_i^{\omega(a_i)} }
=
\sum_{\substack{\mathbf{a}' \in \Z_{\geq 1}^n, \ a_i' \leq X_i \\ \gcd(a_i', d_i) = 1 \\ a_i' \mid \ell^\infty, \ \ell^2 \mid a_1' \cdots a_n'}} 
\prod_{i=1}^n\frac{ h_i(a_i') }{ c_i^{\omega(a_i')} }
\sum_{\substack{\mathbf{a}'' \in \Z_{\geq 1}^n \\ a_i'' \leq X_i/a_i' \\ \gcd(a_i'', d_i\ell) = 1 \\ a_i'a_i'' \equiv \alpha_i \bmod 8}} 
\prod_{i = 1}^n\frac{ h_i(a_i'') }{ c_i^{\omega(a_i'')} }.
\end{equation}
The sum over $\mathbf{a}'$ in~\eqref{eq:asum} is absolutely convergent and we may extend it to infinity. The error in doing so for the variable $a_{i_0}$ is bounded by
\[
X_1 \cdots X_n
\sum_\ell \sum_{\substack{\mathbf{a'} \in \Z_{\geq 1}^n, \ a_{i_0}' \geq X_{i_0}\\ \ell^2 \mid a_1' \cdots a_n', \ a_i' \mid \ell^{\infty}}} 
\frac{ 1 }{ c_1^{\omega(a_1')} \cdots c_n^{\omega(a_n')} a_1' \cdots a_n'}
\ll
X_1 \cdots X_n
\sum_\ell \sum_{\substack{m \geq X_{i_0} \\ \ell^2 \mid m, \ m \mid \ell^\infty}} \frac{\tau_n(m)}{m},
\]
where we used our hypothesis $c_1, \dots, c_n \geq 1$. Applying Rankin's trick, we may bound this by
\[
X_1 \cdots X_n X_{i_0}^{-\frac{1}{3}}
\sum_\ell \sum_{\substack{m \geq X_{i_0} \\ \ell^2 \mid m, \ m \mid \ell^\infty}} \frac{\tau_n(m)}{m^{\frac{2}{3}}}.
\]
By the crude divisor bound $\tau_n(m) \ll m^{\frac{1}{12}}$ and Lemma~\ref{lem:sofos}, we see that the remaining $\ell$ sum is absolutely convergent. 

The $a_i''$ sums in~\eqref{eq:asum} are now all separate so we may apply Lemmas~\ref{lem:LSD} and~\ref{lem:LSD2} as appropriate to get
\[
\prod_{i = 1}^r \widetilde{\beta}_{c_i} \prod_{i=r+1}^n \beta_{c_i}
\prod_{i = 1}^n \frac{ X_i }{(\log X_i)^{1-\gamma_i} }\times \prod_{p>2} F_p(\ell) \times G(\ell)
\times \left( 1 + O \left( (\log \min_i X_i)^{-1}\right) \right),
\]
where
\begin{align*}
F_p(\ell) \coloneqq \prod_{\substack{i\in[n]\\p\mid \ell d_i}}\left( 1 + \frac{h_i(p)}{c_i(p-1)} \right)^{-1}&&\text{and}&&
G(\ell) \coloneqq 
\sum_{\substack{\mathbf{a}' \in \Z_{\geq 1}^n \\ \gcd(a_i', d_i) = 1 \\ \ell^2 \mid a_1' \cdots a_n', \ a_i' \mid \ell^\infty}} 
\prod_{i=1}^n \frac{ h_i(a_i') }{ c_i^{\omega(a_i')}a_i' }.
\end{align*}
Note that above we have made use of the fact that $\log(X/a_i') = \log X \left( 1 + O \left( \frac{\log a_i'}{\log X} \right) \right)$ to replace the log terms by $\log X$ at the cost of the error
\[
\frac{X_1 \cdots X_n}{\min_i \log X_i \prod_{i = 1}^n (\log X_i)^{1-\gamma_i}}
\sum_\ell \sum_{\ell^2 \mid m, \ m \mid \ell^\infty} \frac{\tau_n(m) \log m}{m}.
\]
The $\ell$ sum converges by the divisor bound $\tau_n(m) \log m \ll m^{\frac{1}{4}}$ and Lemma~\ref{lem:sofos} with $\epsilon = 1/4$. For a prime $p > 2$ we have
\[
G(p) = \sum_{k \geq 2} \frac{1}{p^k}
\sum_{\substack{ j_1 + \cdots + j_n = k \\ p\mid d_i\Rightarrow j_i=0}}
\prod_{\substack{ i \in [n] \\ j_i \neq 0}} \frac{h_i(p)}{c_i}.
\]
Hence we obtain
\[
\sum_{\ell \text{ odd}} \mu(\ell) G(\ell) \prod_{p>2} F_p(\ell)
= \prod_{p>2}F_p(p)\left(F_p(p)^{-1}F_p(1)-
\sum_{k\geq 2}\frac{1}{p^k}
\sum_{\substack{ j_1 + \cdots + j_n = k \\ p\mid d_i\Rightarrow j_i=0}}
\prod_{\substack{ i \in [n] \\ j_i \neq 0}} \frac{h_i(p)}{c_i}
\right).
\]
We can expand the term $F_p(p)^{-1}F_p(1)$ as
\[
\prod_{\substack{i\in[n] \\ p\nmid d_i}} \left( 1 + \frac{h_i(p)}{c_i(p-1)} \right) = \prod_{\substack{i \in [n] \\ p\nmid d_i}} \left( 1 + \frac{h_i(p)}{c_i} \sum_{k \geq 1} \frac{1}{p^k}\right) = \sum_{k \geq 0} \frac{1}{p^k}
\sum_{\substack{ j_1 + \cdots + j_n = k \\ p\mid d_i \Rightarrow j_i = 0}}
\prod_{\substack{ i \in [n] \\ j_i \neq 0}} \frac{h_i(p)}{c_i}.
\]
Therefore we can rewrite the previous expression as
\[
\sum_{\ell \text{ odd}} \mu(\ell) G(\ell) \prod_{p>2}F_p(\ell) = \prod_{p>2} \prod_{i\in[n]}\left( 1 + \frac{h_i(p)}{c_i(p-1)} \right)^{-1}\left(1+ \frac{1}{p} \sum_{\substack{i\in [n]\\ p\nmid d_i}}\frac{h_i(p)}{c_i} \right) .
\]
Noting that 
\[
\prod_{i = 1}^r \widetilde{\beta}_{c_i} \prod_{i = r + 1}^n \beta_{c_i} = \frac{1}{\prod_{i = 1}^n\Gamma(\gamma_i)} \times
\frac{1}{2^{r + 2(n - r) + \gamma}}
\prod_{p>2} \prod_{i = 1}^n \left( 1 + \frac{h_i(p)}{c_i(p-1)} \right) \left( 1 - \frac{1}{p}\right)^{\gamma},
\]
we conclude that
\begin{multline*}
\sum_{\substack{\mathbf{a} \in \Z_{\geq 1}^n, \ a_i \leq X_i \\ \gcd(a_i, d_i) = 1 \\ a_i \equiv \alpha_i \bmod 8}} 
\mu^2(a_1 \cdots a_n)\prod_{i=1}^n\frac{ h_i(a_i) }{ c_i^{\omega(a_i)} } 
= \frac{2^{-\gamma-(2n-r)}}{\prod_{i=1}^n\Gamma(\gamma_i)}\times
\prod_{p>2} \left(1+
\frac{1}{p} \sum_{\substack{i\in [n]\\ p\nmid d_i}} \frac{h_i(p)}{c_i} \right)
\left( 1 - \frac{1}{p}\right)^{\gamma}\times\\ \times
\prod_{i = 1}^n \frac{ X_i }{(\log X_i)^{1-\gamma_i} }\times \left( 1 + O \left( (\log \min_i X_i)^{-1}\right) \right)
\end{multline*}
as required.
\end{proof}

\subsection{The final computation}
We will now put all the ingredients together to rewrite the sum~\eqref{eFinalSum} from Theorem \ref{tIntermediate} as a finite sum of main terms. We start by introducing some notation so that we will be able to write the main term from Theorem~\ref{theorem:maincounting} succinctly. We introduce for all sets $\varnothing \subseteq S \subset [n]$, all elements $f$ of $\RAlt$, and all $J\subseteq[n]$
\begin{align*}
\Br(f, S, a) &\coloneqq (-1, a)_2^{|\{-\} \cap f(S)| + \sum_{i \in S} |\{-\} \cap f(\{i\})| + \sum_{\{i, j\} \subseteq S} |\{j\} \cap f(\{i\})|},\\
\BrM(f, J, S, a) &\coloneqq 
\Br(f, S, a) \times (-1, a)_2^{|J \cap S|},
\end{align*}
which are both multiplicative in $a$ and take values in $\pm 1$. We introduce the quantities
\begin{equation}
\label{eq:Tdef}
\Tw(f, \mathbf{a}) \coloneqq
\prod_i (-1, a_i)_2^{|\{-\} \cap f(\{i\})|} \times \prod_{\{i, j\} \subseteq [n]} (a_i, a_j)_2^{|\{j\} \cap f(\{i\})|}
\end{equation}
and
\begin{equation}
\label{eq:hdef}
h_S(p) \coloneqq
\begin{cases}
 \hfil 1& \text{if } p \equiv 1 \bmod 4, \\ 
 \hfil \mathbf{1}_{S \notin \mathcal{D}} &\text{if } p \equiv 3 \bmod 4.
\end{cases} 
\end{equation}
For $i \in [n]$, we will use the abbreviation $h_i(p) \coloneqq h_{\{i\}}(p)$. Set $c_S \coloneqq |W_S|$. By convention, set $c_{\varnothing}\coloneqq 1$ and $h_{\varnothing}(p) = 1$. With these notations, the expression in~\eqref{eFinalSum} can be rewritten as
\begin{multline}
\label{eFinalSum2}
N_2(B, \mathbf{s}, \mathbf{u}, \boldsymbol{\lambda}, \mathbf{b}) = \sum_{f \in \RAlt} \Tw(f,\mathbf{tw})\Tw(f,\mathbf{s}) \times \\
\sum_{\substack{(w_{S})_{|S| \geq 2} \\ w_S = 1 \text{ if } f(S)\notin W_S}}^\flat \frac{\mu^2\left(2 \cdot \prod_{|S| \geq 2} w_{S}\right) \prod_{|S| \geq 2} \Br(f, S, w_S)}{\prod_{|S| \geq 2} c_S^{\omega(w_S)}} 
\times \sum_{(w_i)_{i \in [n]}}^\flat \frac{\mu^2\left(2 \cdot \prod_{i \in [n]} w_i\right)}{\prod_{i \in [n]} c_i^{\omega(w_{i})}}.
\end{multline}
We also introduce the shorthands
\begin{align}
\label{eSigmaDef}
\Sigma(p) \coloneqq
\left(1 + \frac{1}{p} \sum_{j = 1}^n \frac{h_{j}(p)}{c_j}\right)^{-1}&&\text{and}&&\Sigma_m(p) \coloneqq \Sigma(p) 
\left(1 + \frac{1}{p} \sum_{\substack{j\in[n]\\j\neq m}} \frac{h_{j}(p)}{c_j}\right). 
\end{align}
We apply Theorem~\ref{theorem:maincounting} with $\alpha_i = s_iu_i \prod_{\substack{|S| \geq 2 \\ i\in S}} w_S$ to evaluate the inner sum over $(w_i)_{i \in [n]}$ in equation~\eqref{eFinalSum2} as
\begin{multline*}
\frac{\mathbf{F}(d_1, \dots, d_n) B^n (\log B)^{-n + \sum_{i}|V_{\{i\}}|^{-1}}}{\prod_i \Gamma(|V_{\{i\}}|^{-1}) \prod_{|S| \geq 2} w_S^{|S|} \prod_{i = 1}^n b_i^2} \times \frac{1}{2^{2n - |\{i \in [n] : \{i\} \in \mathcal{D}\}| + \sum_i \left(\lambda_i + |V_{\{i\}}|^{-1}\right)}} \times \\
\prod_{p>2}\frac{\left( 1- p^{-1}\right)^{\sum_i|V_{\{i\}}|^{-1}}}{\Sigma(p)} 
\times \prod_{\{i\} \in \mathcal{D}} \frac{1}{2}\left(1+\leg{-1}{s_iu_i\prod_{\substack{|S|\geq 2\\ i\in S}}w_S}\right) \times \left(1 + O\left(\frac{1}{\log B}\right)\right),
\end{multline*} 
where $\mathbf{F}$ is defined in~\eqref{eq:Fdef} and where $d_j = \prod_{|S| \geq 2} w_S \cdot \gcd(\{b_i:i\in[n]-\{j\}\})$. Here the product over $\{i\} \in \mathcal{D}$ detects that $\alpha_i \equiv 1 \bmod 4$ (observe that the sum in Theorem~\ref{theorem:maincounting} is empty if $\alpha_i \not \equiv 1 \bmod 4$ for some $\{i\} \in \mathcal{D}$). Next we expand the product of quadratic residue symbols as
\[
\prod_{\{i\} \in \mathcal{D}} \frac{1}{2} \left(1 + \leg{-1}{s_iu_i\prod_{\substack{|S|\geq 2\\ i\in S}}w_S}\right) = \frac{1}{2^{|\{i \in [n] : \{i\} \in \mathcal{D}\}|}} \hspace{-0.24cm} \sum_{J \subseteq \{i \in [n] : \{i\} \in \mathcal{D}\}} \leg{-1}{\prod_{i\in J} s_iu_i\prod_{|S|\geq 2 }w_S^{|S\cap J|}}. 
\]
We define a new function
$$
g(w, (b_i)_i) = \mathbf{F}(d_1, \dots, d_n) \mu^2\left(2 w\right)
$$ 
with $d_j \coloneqq w \cdot \gcd(\{b_i:i\in[n]-\{j\}\})$ for $j \in [n]$. Then we obtain the asymptotic expansion
\begin{multline}
\label{eFinalSum3}
N_2(B, \mathbf{s}, \mathbf{u}, \boldsymbol{\lambda}, \mathbf{b})
=
\frac{B^n (\log B)^{-n + \sum_i |V_{\{i\}}|^{-1}} \times \left( 1 + O \left( (\log B)^{-1}\right) \right)}{2^{2n + \sum_i \left(\lambda_i + |V_{\{i\}}|^{-1}\right)} \times \prod_i \Gamma(|V_{\{i\}}|^{-1})} \prod_{p>2}\frac{\left( 1- p^{-1}\right)^{\sum_i|V_{\{i\}}|^{-1}}}{\Sigma(p)} \times \\
\sum_{\substack{f \in \RAlt \\ J \subseteq \{i \in [n] : \{i\} \in \mathcal{D}\}}} \hspace{-0.5cm} \Tw(f,\mathbf{tw}) \Tw(f,\mathbf{s}) \prod_{i \in J} \leg{-1}{s_iu_i} \hspace{-0.5cm} \sum_{\substack{(w_{S})_{|S| \geq 2} \\ w_S = 1\text{ if }f(S) \notin W_S}}^{\flat \flat} \hspace{-0.55cm} \frac{ g(\prod_{|S| \geq 2} w_S, (b_i)_i) \prod_{|S| \geq 2} \BrM(f, J, S, w_S)}{\prod_{|S| \geq 2} c_S^{\omega(w_S)} \prod_{|S| \geq 2} w_S^{|S|} \prod_{i = 1}^n b_i^2},
\end{multline}
where the summation conditions $\flat \flat$ are the conditions from $\flat$ pertaining solely the variables $w_S$ with $|S| \geq 2$ (i.e.~$p \mid w_S$ implies $p \equiv 1 \bmod 4$ for all $S \in \mathcal{D}$ and $\gcd(\{w_S\} \cup \{b_i : i \not \in S\}) = 1$ for all $S$). Here we have dropped the condition $w_S \leq (\log B)^{100}$ incurring an acceptable error.

We are now going to spend some time computing
\begin{equation}
\label{eq:bsum}
\sum_{\substack{(b_i)_i \\ \gcd(2^{\lambda_1}b_1, \dots, 2^{\lambda_n}b_n) = 1}} \sum_{\substack{(w_S)_{|S| \geq 2} \\ w_S = 1 \text{ if } f(S) \notin W_S\\ p \mid w_S \Rightarrow p \equiv 1 \bmod 4 \ \forall S \in \mathcal{D} \\ \gcd(\{w_S\} \cup \{b_i : i \notin S\}) = 1}} \frac{g(\prod_{|S| \geq 2} w_S, (b_i)_i) \prod_{|S| \geq 2} \BrM(f, J, S, w_S)}{\prod_{|S| \geq 2} c_S^{\omega(w_S)} \prod_{|S| \geq 2} w_S^{|S|} \prod_{i = 1}^n b_i^2} .
\end{equation}
Note that the sum~\eqref{eq:bsum} is exactly the inner sum of the asymptotic expansion~\eqref{eFinalSum3} of $N_2(B, \mathbf{s}, \mathbf{u}, \boldsymbol{\lambda}, \mathbf{b})$ summed over all $\mathbf{b}$. This sum is readily seen to be absolutely convergent. Once we are done with calculating~\eqref{eq:bsum}, we plug the result back into~\eqref{eFinalSum3} to get an asymptotic expansion for $\sum_\mathbf{b} N_2(B, \mathbf{s}, \mathbf{u}, \boldsymbol{\lambda}, \mathbf{b})$.

\begin{lemma}
\label{lgProperties}
If $\gcd(b_1, \dots, b_n) = 1$, then the function $g$ satisfies 
\[
g(w, (b_i)_i) = \mu^2\left(2 w\right)\prod_{p\mid w}\Sigma(p)\prod_{m=1}^n\prod_{\substack{p>2\\p\mid b_i\forall i\neq m\\p\nmid b_mw}}\Sigma_m(p),\]
where $\Sigma(p)$ and $\Sigma_m(p)$ are defined in~\eqref{eSigmaDef}.
In particular, if $w \prod_{i = 1}^n b_i$ and $w' \prod_{i = 1}^n b_i'$ are coprime, then
\[
g(ww', (b_i b_i')_i) = g(w, (b_i)_i) g(w', (b_i')_i).\]
\end{lemma}

We make a small detour by computing certain geometric series.

\begin{lemma}
\label{lRationalFunctions}
Let $n \geq 1$ be an integer, let $i\in[n]$ and let $\alpha > 0$ be real. Then we have
$$
\sum_{\substack{k_1, \dots, k_n \\ \min(k_1, \dots, k_i) = 0}} \frac{1}{p^{\alpha(k_1 + \dots + k_n)}} = \left(1 - \frac{1}{p^{\alpha i}}\right) \left(\frac{1}{1 - p^{-\alpha}}\right)^n.
$$
\end{lemma}

\begin{proof}
We have
\begin{equation}
\label{eSplitTrick}
\sum_{\substack{k_1, \dots, k_n \\ \min(k_1, \dots, k_i) = 0}} \frac{1}{p^{\alpha(k_1 + \dots + k_n)}} + \sum_{\substack{k_1, \dots, k_n \\ \min(k_1, \dots, k_i) \geq 1}} \frac{1}{p^{\alpha(k_1 + \dots + k_n)}} = \sum_{k_1, \dots, k_n} \frac{1}{p^{\alpha(k_1 + \dots + k_n)}}.
\end{equation}
By dividing through $p^{\alpha i}$ we also see that
$$
\sum_{\substack{k_1, \dots, k_n \\ \min(k_1, \dots, k_i) \geq 1}} \frac{1}{p^{\alpha(k_1 + \dots + k_n)}} = \frac{1}{p^{\alpha i}} \sum_{k_1, \dots, k_n} \frac{1}{p^{\alpha(k_1 + \dots + k_n)}}.
$$
Plugging this back into equation~\eqref{eSplitTrick} we obtain
\[
\sum_{\substack{k_1, \dots, k_n \\ \min(k_1, \dots, k_i) = 0}} \frac{1}{p^{\alpha(k_1 + \dots + k_n)}} = \left(1 - \frac{1}{p^{\alpha i}}\right) \sum_{k_1, \dots, k_n} \frac{1}{p^{\alpha(k_1 + \dots + k_n)}} = \left(1 - \frac{1}{p^{\alpha i}}\right) \left(\frac{1}{1 - p^{-\alpha}}\right)^n
\]
as desired.
\end{proof}

Set
$$
T \coloneqq \sum_{\substack{(k_i)_i \\ \min(k_1+\lambda_1, \dots, k_n+\lambda_n) = 0}} \frac{1}{2^{2 \sum_i k_i}}.
$$
Thanks to Lemma~\ref{lgProperties}, we may rewrite~\eqref{eq:bsum} as an Euler product
$$
T \times \prod_{p > 2} \left(\sum_{\substack{(j_S)_{|S| \geq 2}, \ j_S = 0 \text{ if } f(S) \notin W_S \\ \sum_{|S| \geq 2} j_S\in\{0,1\} \\ p \equiv 3 \bmod 4 \Rightarrow j_S = 0 \ \forall S \in \mathcal{D}}} \sum_{\substack{(k_i)_i \\ \min(k_1, \dots, k_n) = 0 \\ \min(j_S, \{k_i : i \notin S\}) = 0}} \hspace{-0.5cm} \frac{g(\prod_{|S| \geq 2}p^{j_S}, (p^{k_i})_i) \prod_{|S| \geq 2} \BrM(f, J, S, p^{j_S})}{\prod_{|S| \geq 2} c_S^{\mathbf{1}_{j_S > 0}} p^{\sum_{|S| \geq 2} |S| j_S} p^{2 \sum_i k_i}}
\right).
$$
The restriction $\sum_{|S| \geq 2} j_S \in\{0,1\}$ comes exactly from the coprimality of the $w_S$, while the summation conditions $\min(k_1, \dots, k_n) = 0$ and $\min(j_S, \{k_i : i \notin S\}) = 0$ are translations of respectively $\gcd(b_1, \dots, b_n) = 1$ and $\gcd(\{w_S\} \cup \{b_i : i \notin S\}) = 1$. The $2$-adic term $T$ equals
\begin{equation}
\label{eq:twocont}
T = \sum_{(k_i)_i } \frac{1}{2^{2 \sum_i k_i}}-\sum_{\substack{(k_i)_i \\ \lambda_i = 0 \Rightarrow k_i \geq 1}} \frac{1}{2^{2 \sum_i k_i}} = \left(\frac{1}{1 - 2^{-2}}\right)^n \left(1 - \frac{1}{2^{2|\{i \in [n]: \lambda_i = 0\}|}}\right).
\end{equation}
We calculate the various types of contributions to the Euler product individually.

\subsubsection{\texorpdfstring{Case 1: $j_S = 0$ for all $S$}{Case 1}}
First consider the case where $j_S = 0$ for all $S$. For $m \in [n]$, define
\[
E_m \coloneqq \{(k_1, \dots, k_n) \in \Z_{\geq 0}^n : k_m = 0, \ k_j \geq 1\text{ for } j \neq m\}.
\]
By Lemma~\ref{lgProperties}, we have
$$
g(1, (p^{k_i})_i) =
\begin{cases}
\hfil \Sigma_m(p) &\text{if } (k_1, \dots, k_n) \in E_m, \\
\hfil 1 &\text{otherwise.}
\end{cases}
$$ 
Note that $E_1, \dots, E_n$ are pairwise disjoint subsets of $\{(k_1, \dots, k_n) : \min(k_1, \dots, k_n) = 0\}$. Therefore the contribution in this case is
\begin{equation}
\label{eC0}
C_0 \coloneqq \sum_{\substack{k_1, \dots, k_n \\ \min(k_1, \dots, k_n) = 0}} \frac{1}{p^{2(k_1 + \dots + k_n)}} + \sum_{m = 1}^n (\Sigma_m(p) - 1) \sum_{(k_1, \dots, k_n)\in E_m} \frac{1}{p^{2(k_1 + \dots + k_n)}}.
\end{equation}
The first sum in~\eqref{eC0} can be evaluated using Lemma~\ref{lRationalFunctions}, which gives 
\begin{equation}
\label{eGeom0}
\sum_{\substack{k_1, \dots, k_n \\ \min(k_1, \dots, k_n) = 0}} \frac{1}{p^{2(k_1 + \dots + k_n)}} = \left(1 - \frac{1}{p^{2 n}}\right) \frac{1}{(1 - p^{-2})^n}.
\end{equation}
The sum over $(k_1, \dots, k_n) \in E_m$ is a product of $n - 1$ geometric series
\begin{equation}
\label{eGeom1}
\sum_{(k_1, \dots, k_n) \in E_m} \frac{1}{p^{2(k_1 + \dots + k_n)}} = \frac{1}{p^{2 (n - 1)}} \left(\frac{1}{1 - p^{-2}}\right)^{n - 1}.
\end{equation}
Also observe that
\begin{equation}
\label{eGeom2}
\sum_{m = 1}^n \Sigma_m(p) = n - 1 + \Sigma(p).
\end{equation}
Putting equations~\eqref{eC0},~\eqref{eGeom0},~\eqref{eGeom1} and~\eqref{eGeom2} together, we compute
\begin{align*}
\frac{C_0 }{\Sigma(p)}
&= \left(1 - \frac{1}{p^{2 n}}\right) \frac{\Sigma(p)^{-1}}{(1 - p^{-2})^n} - \frac{\Sigma(p)^{-1}-1}{p^{2(n-1)}(1 - p^{-2})^{n - 1}}  \\
&= \frac{1}{(1 - p^{-2})^n}\left( 1 - \frac{1}{p^{2 n}}+\left(1-\frac{1}{p^{2(n - 1)}} \right)\sum_{j = 1}^n \frac{h_{j}(p)}{pc_j}\right).
\end{align*}
Note that $h_j(p) = 1$ implies that $\{j\} \notin \mathcal{D}$. So assuming $h_j(p) = 1$, we have $(-1,p)_2^{|J\cap \{j\}|} = 1$, and thus we conclude that $\BrM(f, J, \{j\}, p) \cdot h_j(p) = h_j(p)$. This allows us to write
\begin{equation}
\label{eCmRewrite}
\frac{C_0 }{\Sigma(p)} = \frac{1}{(1 - p^{-2})^n}\sum_{\substack{\varnothing\subseteq S\subset [n] \\ |S| \in \{0,1\}}} \frac{\BrM(f, J, S, p)h_{S}(p)}{c_S\cdot p^{|S|}}\left(1-\frac{1}{p^{2(n - |S|)}}\right).
\end{equation}

\subsubsection{\texorpdfstring{Case 2: $j_S > 0$ for some $S$}{Case 2}}
Suppose that $\sum_{|S| \geq 2} j_S = 1$. Then there exists a unique $S$ such that $j_S = 1$, and we have $j_T = 0$ for all $T \neq S$. Lemma~\ref{lgProperties} yields
$$
g(p, (p^{k_i})_i) = \Sigma(p).
$$
Therefore, if $f(S) \in W_S$, then the total contribution from $j_S = 1$ is
\begin{equation}
\label{eCS}
C_S \coloneqq \frac{h_S(p)}{(1 - p^{-2})^n} \left(1 - \frac{1}{p^{2 (n - |S|)}}\right) \times \BrM(f, J, S, p) \times \frac{\Sigma(p)}{c_S \cdot p^{|S|}}
\end{equation}
by Lemma~\ref{lRationalFunctions}.

\subsubsection{Conclusion}
The expression in~\eqref{eq:bsum} is
$$
T \times \prod_{p > 2} \left(C_0 + \sum_{\substack{|S| \geq 2\\f(S)\in W_S}} C_S\right),
$$
where $T$, $C_0$ and $C_S$ are defined in respectively equations~\eqref{eq:twocont},~\eqref{eC0} and~\eqref{eCS}. Then the Euler factor at $p$ equals
\begin{multline*}
\frac{(1 - p^{-1})^{\sum_i |V_{\{i\}}|^{-1}}}{\Sigma(p)}\left(C_0 + \sum_{\substack{|S| \geq 2\\f(S)\in W_S}} C_S\right) \\
= \frac{(1 - p^{-1})^{\sum_i |V_{\{i\}}|^{-1}}}{(1 - p^{-2})^n} \sum_{\substack{\varnothing \subseteq S \subset [n] \\ f(S) \in W_S}} \frac{\BrM(f, J, S, p) \cdot h_S(p)}{c_S \cdot p^{|S|}} \left(1 - \frac{1}{p^{2(n - |S|)}}\right),
\end{multline*}
where we used~\eqref{eCmRewrite}. Returning to equation~\eqref{eFinalSum3}, this allows us to evaluate the main term of $\sum_\mathbf{b} N_2(B, \mathbf{s}, \mathbf{u}, \boldsymbol{\lambda}, \mathbf{b})$. We will summarise our work in the form of the following theorem.

\begin{theorem}
\label{tTechnicalHeart}
Let $\mathbf{s}, \mathbf{u}, \boldsymbol{\lambda}$ be given. Then we have as $B \rightarrow \infty$
\begin{multline*}
N(B, \mathbf{s}, \mathbf{u}, \boldsymbol{\lambda}) \sim \frac{B^n (\log B)^{-n + \sum_i |V_{\{i\}}|^{-1}}}{\prod_i \Gamma(|V_{\{i\}}|^{-1})} \times \frac{\left(1 - \frac{1}{2^{2|\{i\in [n]: \lambda_i = 0\}|}}\right)}{2^{2n + \sum_i \left(\lambda_i + |V_{\{i\}}|^{-1}\right)} \times (1 - 2^{-2})^n} \times \\
\sum_{f \in \RAlt} \Tw(f, \mathbf{tw}) \Tw(f, \mathbf{s}) \sum_{J \subseteq \{i \in [n] : \{i\} \in \mathcal{D}\}} \prod_{i\in J} \leg{-1}{s_iu_i} \times \prod_{p>2} \Lambda_p(f,J),
\end{multline*}
where 
\[ 
\Lambda_p(f,J)\coloneqq \frac{\left( 1- p^{-1}\right)^{\sum_i|V_{\{i\}}|^{-1}}}{(1 - p^{-2})^n} \sum_{\substack{\varnothing \subseteq S \subset [n] \\ f(S)\in W_S}} \frac{\BrM(f, J, S, p) \cdot h_S(p)}{c_S \cdot p^{|S|}} \left(1 - \frac{1}{p^{2(n - |S|)}}\right).
\]
In particular, we have
\begin{multline}
\label{eq:finalnonproj}
N(B) \sim \frac{B^n (\log B)^{-n + \sum_i |V_{\{i\}}|^{-1}}}{\prod_i \Gamma(|V_{\{i\}}|^{-1}) \times 2^{ \sum_i |V_{\{i\}}|^{-1}} \times 3^n} \times \hspace{-0.45cm}
\sum_{\substack{f \in \RAlt \\ J \subseteq \{i \in [n] : \{i\} \in \mathcal{D}\}}}
\sum_{\mathbf{s} \textup{ admissible}} \Tw(f, \mathbf{s}) \prod_{i \in J} \leg{-1}{s_i} \times\\
\sum_{(\mathbf{u}, \boldsymbol{\lambda}) \textup{ admissible}}
\left(1 - \frac{1}{2^{2(n-\sum_i \lambda_i)}}\right) \times
\frac{1}{2^{\sum_i \lambda_i}} \times \Tw(f, 2^{\boldsymbol{\lambda}} \mathbf{u}) \times \prod_{i \in J} \leg{-1}{u_i} \times \prod_{p>2} \Lambda_p(f,J).
\end{multline}
\end{theorem}

\subsubsection{Projective statements}
\label{ssProj}
We started with a general affine counting problem in Section~\ref{sChar}, which led to our main counting result in Theorem~\ref{tTechnicalHeart}. However, it turns out that the main term can be slightly simplified in case we had started with a projective counting problem. We shall set up the requisite lemmas here for such a simplification. This will play a key role in comparing Theorem~\ref{tTechnicalHeart} with the Loughran--Rome--Sofos prediction in Section~\ref{sLeading}, which only deals with the projective setting.

For a place $v$ of $\Q$, we define
$$
\Tw_v(f, \mathbf{a}) \coloneqq \prod_{i \in [n]} (-1, a_i)_v^{|\{-\} \cap f(\{i\})|} \times \prod_{\{i, j\} \subseteq [n]} (a_i, a_j)_v^{|\{j\} \cap f(\{i\})|}.
$$

\begin{lemma}
\label{lemma:projcompare}
Suppose that $|g_{i, \{j\}} \cap [n]| \equiv 0 \bmod 2$ for all $i \in [m]$ and all $j \in [n]$. Let $\lambda \in \Q^{\times}$ and let $v$ be a place of $\Q$. Then
\[
\Tw_v(f, \mathbf{a}) \Tw_v(f, \lambda\mathbf{a}) = (-1, \lambda)_v^{\sum_{i \in [n]} |\{-\} \cap f(\{i\})| + \sum_{\{i, j\} \subseteq [n]} |\{j\} \cap f(\{i\})|}.
\]
\end{lemma}

\begin{proof}
By definition we have
\begin{multline*}
\Tw_v(f, \mathbf{a})\Tw_v(f, \lambda\mathbf{a}) =
\prod_{i\in[n]} (-1, a_i)_v^{|\{-\} \cap f(\{i\})|} \times \prod_{\{i, j\} \subseteq [n]} (a_i, a_j)_v^{|\{j\} \cap f(\{i\})|} \\
\times
\prod_{i\in[n]}(-1, a_i\lambda)_v^{|\{-\} \cap f(\{i\})|} \times \prod_{\{i, j\} \subseteq [n]} (a_i\lambda, a_j\lambda)_v^{|\{j\} \cap f(\{i\})|}.
\end{multline*}
We rewrite this as
\begin{align*}
&= (-1, \lambda)_v^{\sum_{i \in [n]} |\{-\} \cap f(\{i\})|} \times \prod_{\{i, j\} \subseteq [n]}
(\lambda, a_j)_v^{|\{j\} \cap f(\{i\})|}(a_i, \lambda)_v^{|\{j\} \cap f(\{i\})|}(\lambda, \lambda)_v^{|\{j\} \cap f(\{i\})|} \\
&= (-1, \lambda)_v^{\sum_{i \in [n]} |\{-\} \cap f(\{i\})|+\sum_{\{i, j\} \subseteq [n]}|\{j\} \cap f(\{i\})|} \times \prod_{\{i, j\} \subseteq [n]} (\lambda, a_ja_i)_v^{|\{j\} \cap f(\{i\})|}.
\end{align*}
It remains to check that 
$$
\prod_{\{i, j\} \subseteq [n]} (\lambda, a_ja_i)_v^{|\{j\} \cap f(\{i\})|} = 1. 
$$
The number of copies of $a_i$ in $\prod_{\{i, j\} \subseteq [n]} (a_ja_i)^{|\{j\} \cap f(\{i\})|}$ equals 
$$
\sum_{j \in [n]} |\{j\} \cap f(\{i\})| = |f(\{i\}) \cap [n]|, 
$$
which is even because $f(\{i\}) \in V_{\{i\}}$ and because $V_{\{i\}}$ is generated by the $g_{j, \{i\}}$ with $j \in [m]$. Hence we conclude that $\prod_{\{i, j\} \subseteq [n]} (a_ja_i)^{|\{j\} \cap f(\{i\})|}$ is a square, and the result follows.
\end{proof}

\begin{lemma}
\label{lemma:projcancellation}
Assume additionally in~\eqref{eDefV} that
\[
\deg M_{i,1} \equiv \deg M_{i,2} \equiv \deg M_{i,3} \bmod 2 \quad \quad \textup{ for every } i \in [m].
\]
Let $f \in \RAlt$ and $J \subseteq \{i \in [n] : \{i\} \in \mathcal{D}\}$. Assume that
\begin{equation}
\label{eq:Jeqcancel}
|J| \not \equiv \sum_{i\in[n]}|\{-\} \cap f(\{i\})|+\sum_{\{i, j\} \subseteq [n]}|\{j\} \cap f(\{i\})|\bmod 2.
\end{equation}
Then we have
\[
\sum_{\mathbf{s} \textup{ admissible}} \Tw(f, \mathbf{s}) \prod_{i \in J} \leg{-1}{s_i} = 0.
\]
\end{lemma}

\begin{proof}
If $\mathbf{s} = (s_1, \dots, s_n)$ is admissible, then the assumption on the degrees of $M_{i, j}$ implies that $-\mathbf{s} = (-s_1, \dots, -s_n)$ is also admissible. We now group any admissible $\mathbf{s}$ with $-\mathbf{s}$. The assumption on the $M_{i, j}$ also ensures that we are allowed to apply Lemma~\ref{lemma:projcompare}. We conclude
\[
\sum_{\delta\in\{\pm 1\}}\Tw(f, \delta\mathbf{s}) \prod_{i \in J}\leg{-1}{\delta s_i}
= \Tw(f, \mathbf{s}) \prod_{i \in J} \leg{-1}{s_i} \sum_{\delta\in\{\pm 1\}}
\delta^{|J|+\sum_{i\in[n]}|\{-\} \cap f(\{i\})|+\sum_{\{i, j\} \subseteq [n]} |\{j\} \cap f(\{i\})|}
\]
which is zero thanks to equation~\eqref{eq:Jeqcancel}.
\end{proof}

\section{The leading constant}
\label{sLeading}
The purpose of the section is to prove Theorem~\ref{thm:main}. In order to do so, we first recall various notions from the paper of Loughran--Rome--Sofos~\cite{LRS}. We will henceforth assume the more stringent assumptions on $M_{i, j}$ stated just before~\eqref{eq:def} (so in particular, $M_{i, 1}$, $M_{i, 2}$ and $M_{i, 3}$ have the same degree).

\subsection{The conjecture}
We begin by recalling the notion of the subordinate Brauer group which plays a critical role in the Loughran--Rome--Sofos prediction and whose connection to character sum methods is one of the key innovations of the present paper.

We recall that for an integral scheme $X$ over a field $k$ of characteristic 0, the Brauer group is defined as $\Br X \coloneqq \Het^2_{\text{{\'e}t}}(X, \mathbb{G}_m)$ and that for regular $X$ one has the Grothendieck purity theorem
\[
0 \rightarrow \Br X \rightarrow \Br k(X) \rightarrow \bigoplus_{D \in X^{(1)}} \Het^1_{\text{{\'e}t}}(k(D), \Q/\Z),
\] 
where the map $\partial_D : \Br k(X) \rightarrow \Het^1_{\text{{\'e}t}}(k(D), \Q/\Z)$ is called the residue map at $D$. The Brauer group induces a pairing $X(\AAA_\Q) \times \Br X \rightarrow \Q/\Z$ on the adelic points of $X$, known as the Brauer--Manin pairing (c.f.~\cite[Section 13.3]{Brauerbook}).

\begin{definition}
\label{dSub}
Let $X$ be a regular integral Noetherian scheme over $\Q$ which admits an ample line bundle. Let $U \subseteq X$ be an open subset and let $\sB \subseteq \Br U$ be a finite collection of Brauer classes. Then a class $ b\in \Br k(X)$ is \emph{subordinate to} $\sB$ if for each codimension one point $D \in X^{(1)}$, the residue $\partial_D(b)$ lies in the set of residues $\partial_D(\langle \sB \rangle)$ of classes in the subgroup of $\Br U$ generated by $\sB$. The set of all subordinate classes is called the subordinate Brauer group and is denoted $\Br_{\Sub}(X, \sB)$.
\end{definition}

For a subset $V \subseteq X(\Q)$, we will denote by $V_\sB$ the set of points $v \in V$ such that $b(v) = 0$ for all $b \in \sB$. Note that we use the convention that $b(v)\neq 0$ at any point $v$ lying on the ramification locus of $b$. We use analogous notation for subsets of $W \subseteq X(\AAA_\Q)$, and we write $W^{\Br_{\Sub}(X, \sB)}$ for those elements of $W$ which are orthogonal under the Brauer--Manin pairing to every element of the subordinate Brauer group.

There is a virtual Artin representation associated to any collection of Brauer classes $\sB \subseteq \Br U$. Namely, let
\[
\Pic_\sB(\overline{X})_\C = \Pic(\overline{X})_\C - \sum_{D \in X^{(1)}} \left( 1 - \frac{1}{\vert\partial_D(\langle\sB\rangle) \vert} \right)\Ind_{\Q_D}^\Q \C.
\]
The Tamagawa measure $\tau_{\sB}$ associated to the problem is then defined as
\[
\Res_{s=1} L(\Pic_\sB(\overline{X})_\C, s) \times \prod_v \lambda_v^{-1}\omega_v
\]
with $\omega_v$ the local Tamagawa measures on $\PPP^{n-1}(\Q_v)$ due to Peyre~\cite[Definition~2.1]{Peyre} and $\lambda_v$ equal to the local factor of $L(\Pic_\sB(\overline{X})_\C, 1)$ for $v$ non-archimedean and 1 otherwise. The cone constant is given by
\begin{equation}
\label{eq:conecst}
\theta(X) \coloneqq \frac{\alpha(X)}{(\rk(\Pic X) - 1)!},
\end{equation}
where $\alpha(X)$ is Peyre's cone constant~\cite[Definition~2.4]{Peyre}. Finally, we recall that for a divisor $D$ in $X$, we have the modified Fujita invariant
\[
\eta(D) \coloneqq \sup \{ t >0: K_X +tD \in C_{\text{eff}}(X)\},
\] 
where $K_X$ denotes the canonical divisor and $C_{\text{eff}}$ denotes the cone of effective divisors on $X$.

\begin{conjecture}[{\cite[Conjecture~3.16]{LRS}}] \label{conj:LRS}
Let $X$ be a smooth, projective, geometrically integral, weak Fano variety over $\Q$. Let $U \subseteq X$ be an open subset and $\sB \subseteq \Br U$ a finite subset. Assume that $\Br X =\Br_1 X$, that $U(\Q)_{\sB} \neq \varnothing$, and that either $\rk(\Pic X) = 1$ or $\Q[U]^\times =\Q^\times$. Then there exists a thin subset $\Omega \subset X(\Q)$ such that
$$
\#\{ x \in U(\Q)_{\sB} : H(x) \leq B, x \notin \Omega \} \sim c_{\textup{LRS}} B(\log B)^{\rk(\Pic X) - \Delta(\sB) - 1}
$$
where
	\begin{align*}
	\Delta(\sB) & = \sum_{D \in X^{(1)}} ( 1 - |\partial_D(\langle\sB\rangle)|^{-1}), \\
		c_{\textup{LRS}} &= \frac{\theta(X) \cdot \lvert\Br_{\Sub}(X,\sB)/\Br \Q\rvert \cdot \tau_{\sB}\left(X(\AAA_\Q)_{\sB}^{\Br_{\Sub}(X, \sB)}\right)}
	{\Gamma(X,\sB)} \cdot 
	\prod_{D \in X^{(1)}} \eta(D)^{1- |\partial_D(\langle\sB\rangle)|^{-1}}, \\
	\Gamma(X,\sB) &= 
	\begin{cases}
		\hfil \prod_{D \in X^{(1)}}\Gamma(|\partial_D(\langle\sB\rangle)|^{-1}) & \textup{if } \rk(\Pic X) = 1, \\
		\hfil \Gamma\left( \rk(\Pic X) - \Delta(\sB)\right) & \textup{if } \Q[U]^\times = \Q^\times.
	\end{cases}
	\end{align*}
\end{conjecture}

\subsection{Comparing Brauer classes and character sums}
In this section, we will compute the various factors in the above expression. Throughout, we take $\mathscr B$ to be the set of quaternion algebras $\{(-M_{i,1}(\boldsymbol{t})M_{i,3}(\boldsymbol{t}), - M_{i,2}(\boldsymbol{t})M_{i,3}(\boldsymbol{t}))_{\Q(\boldsymbol{t})}: i=1, \ldots, m\}$ corresponding to the conics in~\eqref{eq:def}. Our first result allows us to compute the majority of the invariants in the Loughran--Rome--Sofos constant.

\begin{lemma}
\label{lem:firstconj}
We have 
\begin{enumerate}[(i)]
\item $\theta(\PPP^{n - 1}) = \frac{1}{n}$,
\item $\Gamma (\PPP^{n-1}, \sB) = \prod_i \Gamma (\vert V_{\{i\}} \vert^{-1})$,
\item $\Delta(\sB) = n - \sum_i \vert V_{\{i\}} \vert ^{-1}$,
\item $\prod_{D \in X^{(1)}} \eta(D)^{1- |\partial_D(\langle\sB\rangle)|^{-1}} = n^{\Delta(\mathscr B)}$.
\end{enumerate}
\end{lemma}

\begin{proof}
The first point is immediate from~\eqref{eq:conecst}.
As before, we will call upon the analogy between monomials in $\Q(t_1, \ldots, t_n)$ and subsets of $[n]$, so we may view $\sB$ as the set of classes $(-S_{i,1}S_{i,3}, - S_{i,2}S_{i,3})_{\Q(\boldsymbol{t})}$. In order to compute many of the invariants in the leading constant, we need to understand the residues of the Brauer classes associated to the conics along codimension one points. 

Since the coefficients of the conics are all monomials, the Brauer classes can only ramify at the coordinate axes $D_j \coloneqq \{t_j = 0\}$. The residue at $D_j$ of the symbol $(-S_{i,1}S_{i,3}, - S_{i,2}S_{i,3})_{\Q(\boldsymbol{t})}$ can be easily calculated (see e.g.~\cite[Equation~1.16]{Brauerbook}) and is under our identifications precisely equal to $g_{i, \{j\}}$ as defined in Section~\ref{sec:indicator}. Therefore, the set generated by the residues of the defining Brauer classes along the divisors $D_j$ is 
\begin{equation}
\label{eq:residueVi}
\partial_{D_j} ( \langle \sB \rangle) = \langle g_{i, \{j\}} : i =1, \ldots, n \rangle = V_{\{j\}},
\end{equation} 
where we have implicitly identified the residues with subsets of $[n] \cup \{-\}$. From this, part $(ii)$ immediately follows and taking the sum, we get part $(iii)$
\[
\Delta(\sB) = \sum_{j = 1}^n \left( 1 - \left \vert\partial_{D_j} ( \langle \sB \rangle)\right \vert^{-1} \right) = n - \sum_{j=1}^n \vert V_{\{j\}} \vert^{-1},
\]
as desired. Finally, we have $\eta(D_j) = n$ for each $j$ and from this and the previous result, part $(iv)$ follows. 
\end{proof}

This lemma tells us that the order of magnitude of our leading term is correct, as well as the Gamma factors, since $\rk(\Pic \PPP^{n - 1}) = 1$. 

Next, we need to understand the subordinate Brauer group. We write $\langle S, T \rangle$ for the cardinality of $S \cap T$ modulo $2$. Let $\PAlt$ be the subset of $f \in \FAlt$ satisfying
\begin{equation}
\label{eq:weirdcondition}
\sum_{\{i, j\} \subseteq [n]} \langle \{i\}, f(\{j\}) \rangle + \sum_{k \in [n]} \langle \{-\}, f(\{k\}) \rangle = 0.
\end{equation}
Using Lemma~\ref{lemma:projqs} with $S = [n]$, we can readily verify that~\eqref{eq:Jeqcancel} and~\eqref{eq:weirdcondition} are the negation of each other. The relevance of $\FAlt$ is that we have initially started with a general affine problem. However, in the projective setting, as we have seen in Lemma~\ref{lemma:projcancellation}, main terms cancel unless the condition~\eqref{eq:weirdcondition} holds, so the relevant subordinate Brauer group is $\PAlt$.

\begin{lemma}
\label{t:identification}
There is a bijection between elements of $\Br_{\Sub}(\PPP^{n - 1},\sB)/\Br \Q$ and $\PAlt$.
\end{lemma}

\begin{proof}
We will demonstrate an injective map from $\Br_{\Sub}(\PPP^{n - 1}, \sB)$ to $\PAlt$ together with a section from $\PAlt$, thus establishing the claim. 

In the first direction, consider an element $\alpha \in \Br_{\Sub}(\PPP^{n - 1}, \sB)$. By definition of subordinance, $\partial_D(\alpha) \in\partial_D\left( \langle\sB \rangle\right)$ for any codimension one point $D$. If $D$ is not a component of $\{t_1 \cdots t_n =0\}$, then $\partial_D\left(\langle\sB \rangle\right)$ only contains 0. Hence, $\alpha$ can only ramify along the coordinate axes. Suppose $D$ is one of these axes and write $E = f^{-1}(D)$ for the inverse image under the fibration map $f: V_1 \times \cdots \times V_m \rightarrow \PPP^{n - 1}$, where the $V_i$ are the conic bundles associated to the classes in $\sB$ and where $\times$ denotes the fibre product over $\PPP^{n - 1}$. Then $E = E_1 \times \cdots \times E_m$ and the residue field $\kappa_E$ is the compositum of the residue fields $\kappa_{E_i}$. Each $E_i$ corresponds either to a smooth conic over the point $D$, in which case $\kappa_{E_i}$ equals $\Q(D)$ or else a pair of lines in which case $\kappa_{E_i}$ is a quadratic extension of $\Q(D)$. Hence, $\kappa_E$ is a multiquadratic extension and combined with~\cite[Lemma~2.4]{LRS} and~\cite[Lemma~10.1.3]{Brauerbook}, we conclude that $\alpha$ must be 2-torsion.

This means that $\alpha$ can be written as a linear combination of quaternion algebras of the form $(t_i/t_n, t_j/t_n)$ or $(c, t_i/t_n)$ for $c \in \Q^\times/\Q^{\times 2}$, since these generate the non-trivial 2-torsion elements in $\Br(\Q(\PPP^{n - 1}))$ with non-trivial ramification only above $\{t_1\cdots t_n = 0\}$. Note that only the scalar $c=-1$ is relevant here as this is the only scalar that can appear in the residues since the $M_{i, j}$ are monic.

We will associate $\alpha$ to the map $f_\alpha$ which we define as follows. Define $f_{\alpha}: \mathcal{P}([n]) \rightarrow \mathcal{P}([n] \cup \{-\})$ through
\[
\partial_{t_i} \alpha = \prod_{j\in f_{\alpha}(\{i\})} t_j,
\]
where $t_{\{-\}} = -1$ by convention. We claim that this gives rise to a unique element of $\FAlt$.

We now check the alternating properties. Firstly, it is an immediate consequence of the definition that $t_i \nmid \partial_{t_i} \alpha$ and hence $\{i\} \cap f_\alpha(\{i\}) = \varnothing$. Observe that by construction 
\[
\vert \{i\} \cap f_\alpha(\{j\})\vert
= \begin{cases}
 \hfil 1 &\text{if } t_i \mid \partial_{t_j} \alpha\\
 \hfil 0 &\text{otherwise.}
\end{cases}
\] 
Write $\alpha = \sum_{i < j} e_{i, j} (t_i/t_n, t_j/t_n) + \sum_k c_k (-1,t_k/t_n)$ and $e_{i, j} \coloneqq e_{j, i}$ for $i > j$. We compute
$$
\partial_{t_j} \alpha = 
\begin{cases} \hfil 
(-1)^{c_j}t_n^{\sum_{i \in [n - 1] - \{j\}} e_{i, j}} \prod_{i \in [n - 1]-\{ j\}} t_i^{e_{i,j}} &\text{if }j \in [n - 1]\\
\prod_{\{i_1, i_2\} \subseteq [n - 1]} (-t_{i_1} t_{i_2})^{e_{i_1, i_2}} \times \prod_{k \in [n - 1]} (-1)^{c_k} &\text{if }j = n.
\end{cases}
$$
Hence, for $\{i, j\} \subseteq [n - 1]$, we have
$$
t_i \mid \partial_{t_j} \alpha \iff e_{i, j} \neq 0 \iff t_j \mid \partial_{t_i} \alpha.
$$ 
For $i=n$, we deduce similarly
$$
t_n \mid \partial_{t_j} \alpha \iff \sum_{i \in [n - 1] - \{j\}} e_{i,j}\equiv 1\bmod 2 \iff t_j \mid \partial_{t_n} \alpha.
$$ 
Therefore $\langle \{i\}, f_\alpha(\{j\}) \rangle = \langle \{j\}, f_\alpha(\{i\}) \rangle$. By~\eqref{eq:residueVi} and by construction of $f_\alpha$, we have that 
$$
f_\alpha(\{i\}) \in V_{\{i\}}
$$ 
for all $i \in [n]$, establishing the claim. Finally, we must check the last condition~\eqref{eq:weirdcondition}. 

We observe that since $\alpha$ is subordinate, we must have 
\[
\partial_{t_n} \alpha = \prod_{\{i, j\} \subseteq [n - 1]} (-t_i t_j)^{e_{i, j}} \times \prod_{k \in [n - 1]} (-1)^{c_k} \in \partial_{t_n}(\langle \sB \rangle).
\] 
Translating this into the language of subsets, the $i$ for which $t_i$ arises in this residue are precisely those for which $\sum_{j \in [n - 1] - \{i\}} e_{i, j}$ is odd and the exponent of $-1$ in the residue is precisely $\sum_{\{i, j\}\subseteq [n - 1]} e_{i, j} + \sum_{k \in [n - 1]} c_k$. Observe that
\[ \langle \{-\}, f(\{n\}) \rangle 
=\sum_{\{i, j\}\subseteq [n - 1]} e_{i, j} + \sum_{k \in [n - 1]} c_k
=
\sum_{\{i, j\} \subseteq [n-1]} \langle \{i\}, f(\{j\}) \rangle + \sum_{k \in [n-1]} \langle \{-\}, f(\{k\}) \rangle .
\]
Finally, we see that
\[
\sum_{i \in [n - 1]} \langle \{i\}, f(\{n\}) \rangle = \sum_{\substack{i, j \in [n - 1] \\ i \neq j}} e_{i, j} = 0. 
\]
Summing these two equations, we readily verify~\eqref{eq:weirdcondition}.

Thus we have found an homomorphism $\Br_{\Sub}(\PPP^{n - 1}, \sB) \rightarrow \PAlt$ by sending $\alpha$ to $f_{\alpha}$. Let $\alpha \in \Br_{\Sub}(\PPP^{n-1}, \sB)$ and suppose that $f_{\alpha}$ is the trivial element of $\PAlt$. Then $f_\alpha(\{i\}) = \varnothing$ and thus $\partial_{t_i}(\alpha) = 1$ for all $i \in [n-1]$. We conclude that $\alpha$ can ramify at most at $\{t_n=0\}$ and thus $\alpha \in \Br(\PPP^{n - 1}/\{t_n = 0\}) = \Br(\AAA^{n - 1}) = \Br \Q$. Hence we get an injective map from $\Br_{\Sub}(\PPP^{n-1}, \sB)/\Br \Q$ to $\PAlt$.

For the reverse direction, let $f \in \PAlt$ and define 
\[
\alpha_f = \sum_{\{i,j\}\subseteq [n - 1]} \langle \{i\}, f(\{j\}) \rangle (t_i/t_n, t_j/t_n) + \sum_{k\in[n-1]}\langle \{-\}, f(\{k\}) \rangle (-1, t_k/t_n) ,
\] 
where we recall that the inner product is the size of the intersection modulo 2. Since it is a sum of quaternion algebras which live in $\Br(\Q(\PPP^{n - 1}))$, so does $\alpha_f$. To see that $\alpha_f$ is subordinate, we compute the residue along $\{t_r = 0\}$ for $r \in [n - 1]$
\begin{align*}
\partial_{t_r} \alpha_f 
&= (-1)^{\langle \{-\}, f(\{r\}) \rangle} t_n^{\sum_{i \in [n - 1]} \langle \{i\}, f(\{r\})\rangle} \prod_{i \in [n - 1]} t_i^{\langle \{i\}, f(\{r\}) \rangle} \\
&= (-1)^{\langle \{-\}, f(\{r\}) \rangle} t_n^{\sum_{i \in [n]} \langle \{i\}, f(\{r\})\rangle} \prod_{i \in f(\{r\})} t_i.
\end{align*}
Since $f(\{r\}) \in V_{\{r\}}$ and since the generators $g_{i, \{r\}}$ of $V_{\{r\}}$ satisfy $|g_{i, \{r\}} \cap [n]| \equiv 0 \bmod 2$, we have the identity $\sum_{i \in [n]} \langle \{i\}, f(\{r\}) \rangle = 0$, which yields
\[
\partial_{t_r} \alpha_f = (-1)^{\langle \{-\}, f(\{r\}) \rangle} \prod_{i \in f(\{r\})} t_i.
\]
The final thing to check is the residue along the divisor $\{t_n = 0\}$, which is given by
\begin{align*}
\partial_{t_n}\alpha_f&=
\prod_{\{i, j\} \subseteq [n - 1]} (-t_it_j)^{\langle \{i\}, f(\{j\}) \rangle} \times \prod_{k \in [n - 1]} (-1)^{\langle \{-\}, f(\{k\})\rangle} \\
&= (-1)^{\sum_{\{i, j\} \subseteq [n - 1]} \langle \{i\}, f(\{j\}) \rangle + \sum_{k \in [n - 1]} \langle \{-\}, f(\{k\}) \rangle} \prod_{\{i, j\} \subseteq [n - 1]} (t_it_j)^{\langle \{i\}, f(\{j\}) \rangle},
\end{align*} 
where the product is understood to take place modulo squares. We deduce from~\eqref{eq:weirdcondition} that
\begin{align*}
\partial_{t_n}\alpha_f&= (-1)^{ \langle\{-\}, f(\{n\})\rangle + \sum_{i \in [n - 1]} \langle \{i\}, f(\{n\}) \rangle } \prod_{\{i, j\} \subseteq [n - 1]} (t_it_j)^{\langle \{i\}, f(\{j\}) \rangle}.
\end{align*} 
Since $f(\{n\}) \in V_{\{n\}}$ and since the generators $g_{i, \{n\}}$ of $V_{\{n\}}$ satisfy $|g_{i, \{n\}} \cap [n]| \equiv 0 \bmod 2$ (and $n \not \in g_{i, \{n\}}$), we continue rewriting this as
$$
\partial_{t_n}\alpha_f = (-1)^{\langle\{-\}, f(\{n\})\rangle } \prod_{\{i, j\} \subseteq [n - 1]} (t_it_j)^{\langle \{i\}, f(\{j\}) \rangle}.
$$
Finally, we observe that 
$$
\prod_{\{i, j\} \subseteq [n - 1]} (t_it_j)^{\langle \{i\}, f(\{j\}) \rangle} \prod_{j \in [n]} t_j^{\langle \{n\}, f(\{j\}) \rangle} = \prod_{j \in [n - 1]} t_j^{\langle [n], f(\{j\}) \rangle}
$$
is a square using once more that the generators $g_{i, \{j\}}$ of $V_{\{j\}}$ satisfy $|g_{i, \{j\}} \cap [n]| \equiv 0 \bmod 2$. We conclude that
$$
\partial_{t_n}\alpha_f = (-1)^{\langle\{-\}, f(\{n\})\rangle} \prod_{j \in [n]} t_j^{\langle \{n\}, f(\{j\}) \rangle} = (-1)^{\langle\{-\}, f(\{n\})\rangle} \prod_{i \in [n]} t_i^{\langle \{i\}, f(\{n\}) \rangle},
$$
so $\partial_{t_n} \alpha_f = f(\{n\}) \in V_{\{n\}}$. Therefore, for $f\in \PAlt$, we conclude that $\alpha_f \in \Br_{\Sub}(\PPP^{n - 1}, \sB)$. Moreover, $f_{\alpha_f} = f$ so we have a section $s: \PAlt \rightarrow \Br_{\Sub}(\PPP^{n - 1}, \sB)$ of our assignment $\Br_{\Sub}(\PPP^{n - 1}, \sB) \rightarrow \PAlt$. The claimed isomorphism follows.
\end{proof}

We have formed a bridge between the (representatives for) the non-constant classes in the subordinate Brauer group and elements of $\FAlt$ satisfying condition~\eqref{eq:weirdcondition}. In the remainder of this subsection, we will see how that connection allows us to compute the Tamagawa measure. Recall that we may parametrise the elements of $\Z_ p^n$ via
\[ 
t_i = p^{\lambda_i} u_i,
\]
where $\lambda_i \in \Z_{\geq 0}$ and $u_i$ a unit.

\begin{lemma}
\label{lemma:HScond1}
Let $p$ be an odd prime. Let $\varnothing \subseteq T \subset [n]$. Let $t_i = p^{\lambda_i} u_i$. Assume that $T$ equals $\{i \in [n] : \lambda_i \equiv 1 \bmod 2\}$. Then we have the identity of Hilbert symbols
\begin{equation}
\label{eq:titj}
\prod_{\{i, j\} \subseteq [n]} (t_i, t_j)_p^{|\{i\} \cap f(\{j\})|} = \prod_{\{i, j\} \subseteq T} (-1, p)_p^{|\{i\} \cap f(\{j\})|} \prod_{i \in [n]} (u_i, p)_p^{|\{i\} \cap f(T)|}.
\end{equation}
\end{lemma}

\begin{proof}
The Hilbert symbol $(a, b)_p$ is trivial if $p \nmid ab$. Therefore we have for $i \neq j$
\begin{equation}
\label{eq:localtij} (t_i, t_j)_p = 
\begin{cases}
\hfil (-u_iu_j, p)_p&\text{if }i \in T\text{ and }j \in T,\\
\hfil (p, u_j)_p&\text{if }i \in T\text{ and }j\notin T,\\
\hfil (u_i, p)_p&\text{if }i\notin T\text{ and }j \in T,\\
\hfil 1&\text{if }i\notin T\text{ and }j\notin T.\\
\end{cases}
\end{equation}
We apply~\eqref{eq:localtij} to our product of local Hilbert symbols in order to calculate that
\begin{align*}
\prod_{\{i, j\} \subseteq [n]} (t_i, t_j)_p^{|\{i\} \cap f(\{j\})|} &=
\prod_{\{i, j\} \subseteq T}(-u_iu_j, p)_p^{|\{i\} \cap f(\{j\})|} \cdot \prod_{i \notin T} \prod_{j \in T} (u_i, p)_p^{|\{i\} \cap f(\{j\})|}\\
&= \prod_{\{i, j\} \subseteq T} (-1, p)_p^{|\{i\} \cap f(\{j\})|} \cdot \prod_{i \in T} (u_i, p)_p^{|\{i\} \cap f(T)|}\cdot \prod_{i\notin T} (u_i, p)_p^{|\{i\} \cap f(T)|}\\
&= \prod_{\{i, j\} \subseteq T} (-1, p)_p^{|\{i\} \cap f(\{j\})|} \prod_{i \in [n]} (u_i, p)_p^{|\{i\} \cap f(T)|},
\end{align*}
which is exactly equation~\eqref{eq:titj}.
\end{proof}

\begin{lemma}
Given $f \in \PAlt$, let $\alpha_f$ be the corresponding element in $\Br_{\Sub}(\PPP^{n - 1}, \sB)/\Br \Q$ under the bijection given by Lemma~\ref{t:identification}. Then
$$
\tau_{\sB}\left(\PPP^{n - 1}(\AAA_\Q)_{\sB}^{\Br_{\Sub}(\PPP^{n - 1}, \sB)}\right)
=
\frac{1}{\lvert\Br_{\Sub}(\PPP^{n - 1}, \sB)/\Br \Q \rvert} \sum_{f\in\PAlt} \prod_v \sigma_v(f),
$$
where the product runs over all places $v$ of $\Q$, where $e(z) \coloneqq e^{2\pi i z}$ and where
\[
\sigma_v(f) \coloneqq \lambda_v^{-1}\int_{\PPP^{n - 1}(\Q_v)} e\left(\inv_v \alpha_f(x)\right) \mathrm{d}x \, \textup{ and } \, \lambda_v \coloneqq
\begin{cases}
\left( 1 - v^{-1}\right)^{n - 1 - \sum_i \vert V_{\{i\}} \vert^{-1}} &\textup{if $v$ is finite},\\
\hfil 1 &\textup{if $v=\infty$}.
\end{cases}
\]
\end{lemma}

\begin{proof}
See \cite[Lemma~3.13]{LRS}.
\end{proof}

\subsection{Computation of local factors}
\begin{lemma}
\label{lemma:invtoT}
Let $f\in\PAlt$ and $\alpha\in\Br_{\Sub}(\PPP^{n - 1}, \sB)/\Br \Q$ be corresponding elements under Lemma~\ref{t:identification}. Suppose $x\in\PPP^{n - 1}(\Q_v)_{\sB}$. Let $\mathbf{x}$ be any lift of $x$ to the affine cone. Then 
\[
\inv_v\alpha(x)
=\begin{cases}
\hfil 0&\textup{if }\Tw_v(f,\mathbf{x})=1,\\
\hfil \frac{1}{2}&\textup{if }\Tw_v(f,\mathbf{x})=-1.
\end{cases}
\]
\end{lemma}

\begin{proof}
It follows from Lemma~\ref{lemma:projcompare} and~\eqref{eq:weirdcondition} that $\Tw_v(f, \mathbf{x})$ does not depend on the choice of $\mathbf{x}$ above $x$. The equality is then clear by manipulating the product of Hilbert symbols.
\end{proof}

\begin{prop}
\label{lTama}
For every $f \in \PAlt$ and every odd prime $p$, we have
\[
\sigma_p(f)=
\frac{\left( 1 -p^{-1}\right)^{ \sum_i|V_{\{i\}}|^{-1}}}{\left( 1 - p^{-2}\right)^n } 
\sum_{\substack{ \varnothing\subseteq S \subset [n] \\ f(S) \in V_S+\langle \{-\}\rangle}} \frac{h_S(p)}{\lvert W_S \rvert\cdot p^{|S|}} \left(1 - \frac{1}{p^{2 (n-|S|)}}\right)\leg{-1}{p}^{\kappa(f,S)},
\]
where $h_S$ is as defined in~\eqref{eq:hdef} and where
\begin{equation}
\label{eq:kappaS}
\kappa(f, S) \coloneqq \sum_{k \in S} \langle \{-\}, f(\{k\})\rangle + \sum_{\{i, j\} \subseteq S} \langle \{i\}, f(\{j\}) \rangle + \langle \{-\}, f'(S)\rangle
\end{equation}
and $f'$ is the image of $f$ in $\RAlt$ (through $\Alt$ and the section $s$ from~\eqref{eq:Altses}).
\end{prop}

\begin{proof}
By definition, the local integral in $\sigma_p(f)$ equals
\begin{equation}
\label{eq:breakintegral}
\int_{\PPP^{n - 1}(\Q_p)} e(\inv_p \alpha(x)) \mathrm{d}x= \omega_p\left(\left\{ x : \inv_p \alpha(x) = 0\right\}\right) - \omega_p\left(\left\{ x : \inv_p \alpha(x) = \textstyle\frac{1}{2}\right\}\right),
\end{equation}
where $\alpha = \alpha_f$ as in the proof of Theorem~\ref{t:identification}. Let us denote the affine cone of the set $\PPP^{n - 1}(\Q_p)_{\sB}$ by $\wO$, then applying the formula~\cite[Proposition 4.1]{LRS} and Lemma~\ref{lemma:invtoT}, we have 
\[
\omega_p\left(\left\{ x \in \PPP^{n - 1}(\Q_p)_{\sB} : \inv_p \alpha (x) = c \right\} \right)
=
\frac{\mu_p \left( \left\{ \mathbf x \in \wO : p \nmid \mathbf x,\ \Tw_p(f,\mathbf{x})= (-1)^{\iota(c)}\right \}\right)}{1 - p^{-1}},
\]
where $\iota$ is the unique identification between $\frac{1}{2}\Z/\Z$ and $\Z/2\Z$ and $\mu_p$ is the unique Haar measure on $\Q_p^n$ normalised so that $\mu_p(\Z_p^n) = 1$. We pull out the $p$-adic valuation of each point, while also fixing the parity of the valuation. Then we find that $\omega_p(\{ x \in \PPP^{n-1}(\Q_p)_{\sB} : \inv_p \alpha (x) = c\})$ is equal to 
\begin{align*}
&\frac{1}{1-p^{-1}} \sum_{\substack{\boldsymbol{\lambda} \in \{0,1\}^n \\ \boldsymbol{\lambda} \neq \boldsymbol{1}}}
\sum_{\substack{ \mathbf c \in \Z_{\geq 0}^n\\ \mathbf{c} \equiv \boldsymbol{\lambda} \bmod 2 \\ \min(\mathbf c) = 0}} \frac{1}{p^{\sum_i c_i}} \times 
\mu_p \left( \left\{ \mathbf u \in \left(\Z_p^\times\right)^n :p^{ \mathbf{c}} \mathbf{u} \in \wO,\ \Tw_p(f,p^{ \boldsymbol{\lambda}} \mathbf{u})=(-1)^{\iota(c)}\right \}\right)
\\
&= 
\frac{1}{(1-p^{-1})(1-p^{-2})^n} 
\sum_{\substack{\boldsymbol{\lambda} \in \{0,1\}^n \\ \boldsymbol{\lambda} \neq \boldsymbol{1}}} \frac{1}{p^{\sum_i \lambda_i}} \times \left(1 - \frac{1}{p^{2 |\{i \in [n] : \lambda_i = 0\}|}}\right) \times \\
&\hspace{15em}\mu_p \left( \left\{ \mathbf u \in \left(\Z_p^\times\right)^n :p^{\boldsymbol{\lambda}} \mathbf{u} \in \wO,\ \Tw_p(f,p^{ \boldsymbol{\lambda}} \mathbf{u}) = (-1)^{\iota(c)}\right \}\right)
\end{align*}
by Lemma~\ref{lRationalFunctions}. Now we identify each choice of $\boldsymbol \lambda$ with the subset $S = \{i \in[n] : \lambda_i = 1\}$. Under this identification, we see that the condition $(p^{\lambda_1}u_1, \ldots, p^{\lambda_n}u_n) \in \wO$ is equivalent to 
\begin{equation}
\label{eq:originalg}
\left(\frac{\prod_{j \in g_{i, S}} u_j}{p}\right) =1
\end{equation}
for all $i \in [m]$, where $u_{\{-\}} = -1$ by convention. The condition imposed on this set can be understood via a linear system of equations in $\FF_2$. Indeed, set
\[
a_j =
\begin{cases}
 \hfil 0 &\text{ if } u_j \text{ is a quadratic residue mod }p,\\
 \hfil 1 &\text{ otherwise}
\end{cases}
\] 
for $j \in [n]$. Then~\eqref{eq:originalg} is equivalent to 
\begin{equation}
\label{eEquations}
\sum_{j \in \pi(g_{i, S})} a_j = \langle \{-\}, g_{i,S} \rangle\mathbf{1}_{p \equiv 3 \bmod 4}
\end{equation}
for all $i \in [m]$, where $\pi: \mathcal{P}([n] \cup \{-\}) \rightarrow \mathcal{P}([n])$ is the natural projection map. By Lemma~\ref{lemma:HScond1}, the condition $\Tw_p(f,p^{ \boldsymbol{\lambda}} \mathbf{u}) = (-1)^{\iota(c)}$ can be rewritten as
\[
\left(\frac{-1}{p} \right)^{\sum_{\{i, j\} \subseteq S} \langle \{i\}, f(\{j\}) \rangle + \sum_{k \in S} \langle \{-\}, f(\{k\}) \rangle}\times \prod_{i \in \pi(f(S))} \left(\frac{u_i}{p} \right) = (-1)^{\iota(c)}.
\]
We re-express this in terms of linear algebra by viewing this as an equation
\begin{equation}
\label{eLinearAlgebra}
\sum_{j \in \pi(f(S))} a_j = \iota(c) + \mathbf{1}_{p \equiv 3 \bmod 4} \times \left(\sum_{\{i, j\} \subseteq S}\langle \{i\}, f(\{j\}) \rangle + \sum_{k \in S} \langle \{-\}, f(\{k\}) \rangle\right).
\end{equation}

We now distinguish two cases. Firstly, if $f(S)$ is not in the space generated by $V_S$ and $\{-\}$, we see that equation~\eqref{eLinearAlgebra} is independent from the previous equations~\eqref{eEquations}. We conclude that irregardless of $\iota(c) \in \Z/2\Z$ the number of solutions is the same (which could be zero, if the initial equations~\eqref{eEquations} had no solutions). Hence looking at~\eqref{eq:breakintegral}, we see that the Tamagawa number vanishes. Hence only those $S$ for which $f(S) \in V_S + \langle - \rangle$ contribute.

Next we consider the case when $ f(S) \in V_S+\langle \{-\}\rangle$. Assuming~\eqref{eEquations} is soluble, we claim~\eqref{eLinearAlgebra} is equivalent to 
\begin{equation}
\label{eq:nontrivialvalue}
(-1)^{\iota(c)} = \leg{-1}{p}^{\kappa(f,S)}.
\end{equation}
Since $\pi(f(S))$ is a linear combination of $\pi(g_{i,S})$, by substituting~\eqref{eEquations} into~\eqref{eLinearAlgebra}, we see that the system can only be soluble for at most one $\iota(c)\in\Z/2\Z$. When $p\equiv 1\bmod 4$, this is when $\iota(c)=0$, which is equivalent to~\eqref{eq:nontrivialvalue}. When $p \equiv 3 \bmod 4$, write $f'(S) = \sum_{i=1}^m \mu_i g_{i,S}$, so
\[
\sum_{j \in \pi(f(S))} a_j 
= \sum_{i=1}^m \mu_i \sum_{j\in \pi(g_{i,S})}a_j
= \sum_{i=1}^m \mu_i \langle \{-\}, g_{i,S} \rangle
= \langle \{-\}, f'(S) \rangle,
\] 
where the middle equality uses~\eqref{eEquations}. Plugging this into~\eqref{eLinearAlgebra} yields~\eqref{eq:nontrivialvalue}.

It remains to compute the number of solutions to~\eqref{eEquations}, which we will achieve using the indicator function 
\[
\frac{1}{2^m} \prod_{i = 1}^m \left(1 + \left(\frac{\prod_{j \in g_{i, S}} u_j}{p}\right)\right) 
= \frac{\mathbf{1}_{S \in \mathcal{D} \Rightarrow p \equiv 1 \bmod 4}}{|W_S|} \sum_{T \in W_S} \left(\frac{\prod_{j \in T} u_j}{p}\right)
\]
from Lemma~\ref{lExpand}. With this in hand, we compute 
\begin{multline*}
\mu_p \left( \left\{ \mathbf u \in \left(\Z_p^\times\right)^n :p^{\boldsymbol{\lambda}} \mathbf{u} \in \wO\right \}\right)
= \frac{1}{p^n} \sum_{\mathbf u \in \left((\Z_p/p\Z_p)^\times\right)^n} \frac{\mathbf{1}_{S \in \mathcal{D} \Rightarrow p \equiv 1 \bmod 4}}{|W_S|} \sum_{T \in W_S} \left(\frac{\prod_{j \in T} u_j}{p}\right)\\
= \frac{\mathbf{1}_{S \in \mathcal{D} \Rightarrow p \equiv 1 \bmod 4}}{p^n \cdot |W_S|} \sum_{T \in W_S} (p - 1)^{n - |\pi(T)|} \leg{-1}{p}^{|\{-\}\cap T|}\prod_{j \in \pi(T)}\sum_{u_j \in (\Z_p/p\Z_p)^\times} \left(\frac{ u_j}{p}\right).
\end{multline*}
The inner sum over $u_j$ is $0$, so the only terms that contribute are $T = \varnothing$ or $T = \{-\}$. By definition, $W_S$ does not contain $\{-\}$, so we are left with $T = \varnothing$. Therefore we have
\[
\mu_p \left( \left\{ \mathbf u \in \left(\Z_p^\times\right)^n :p^{\boldsymbol{\lambda}} \mathbf{u} \in \wO\right \}\right)
= \left(1 - \frac{1}{p}\right)^n\frac{ \mathbf{1}_{S \in \mathcal{D} \Rightarrow p \equiv 1 \bmod 4}}{|W_S|} \\
= \left(1 - \frac{1}{p}\right)^n \frac{h_S(p)}{|W_S|}.
\]
Putting everything together, we have shown that $\omega_p(\{ x \in \PPP^{n-1}(\Q_p)_{\sB} : \inv_p \alpha (x) = c\})$ equals
\begin{equation}
\label{eq:finalomega}
\frac{\left(1-p^{-1}\right)^{n - 1}}{(1 - p^{-1}) (1 - p^{-2})^n} 
\sum_{\substack{ \varnothing\subseteq S \subset [n] \\ f(S) \in V_S+\langle \{-\}\rangle}} \frac{h_S(p)}{\lvert W_S \rvert\cdot p^{|S|}} \left(1 - \frac{1}{p^{2 (n-|S|)}}\right) \mathbf{1}_{S \text{ satisfies \eqref{eq:nontrivialvalue}}}.
\end{equation}
Returning to~\eqref{eq:breakintegral}, we see that
\[
\sigma_p(f) = \lambda_p^{-1}\sum_{c\in \{0,\frac{1}{2}\}}(-1)^{\iota(c)}\times \omega_p\left(\{x\in \PPP^{n-1}(\Q_p)_{\sB} : \inv_p \alpha(x) = c\}\right),
\]
which completes the proof on substituting~\eqref{eq:nontrivialvalue} and~\eqref{eq:finalomega}.
\end{proof}

\begin{remark}
In the case of the family of diagonal planar conics, we have $n=3$, $V_{\{i\}} = \{ \varnothing, \{-,j,k\}\}$ and $V_{\{i,j\}} = \{ \varnothing, \{i,j\}\} = W_{\{i,j\}}$ for every choice of $\{i,j,k\} = \{1,2,3\}$. Therefore the conditions $f_\alpha(S) \in V_S+\langle \{-\}\rangle$ and $\{-\} \not \in V_S$ are guaranteed for all $S$ and all $f_\alpha$. Moreover, one can see from the definition of $c$ that $c \equiv 0 \bmod 2$ regardless of the choice of $S$ or $f_\alpha$. Therefore the predicted local term at an odd prime $p$ is
\begin{multline*}
\frac{\left( 1 - \frac{1}{p} \right)^{\frac{3}{2}}}{\left( 1 - \frac{1}{p^2} \right)^3}\left[ \left(1 - \frac{1}{p^6}\right) + \frac{1}{p}\times \left( 1 - \frac{1}{p^4}\right) \times \frac{3}{2} + \frac{1}{p^2} \times \left( 1 - \frac{1}{p^2} \right) \times \frac{3}{2} \right] \\
=
\left( 1 - \frac{1}{p}\right)^{\frac{3}{2}} \frac{(p^2+p+1)(2p^2+p+2)}{2(p^2-1)^2},
\end{multline*}
which recovers precisely the local term computed in~\cite[Lemma~5.24]{LRS}.
\end{remark}

We first complete the calculation of the main term predicted by Loughran--Rome--Sofos by investigating the local term at 2 and $\infty$.
Recall the definition of admissibility from Definition~\ref{def:admissible}.

\begin{lemma}
\label{lem:2inf}
For every $f \in \PAlt$, we have 
\begin{align*}
\sigma_\infty(f) &= \frac{n}{2}\sum_{\mathbf s \textup{ admissible} }\Tw(f,\mathbf s),\\
\sigma_2(f) &= \frac{2^{-\sum_i \vert V_{\{i\}}\vert^{-1}}}{3^n}\sum_{(\mathbf{u},\boldsymbol{\lambda} )\textup{ admissible}} \left(1 - \frac{1}{4^{|\{i \in [n] : \lambda_i = 0\}|}}\right)\times \frac{1}{2^{\sum_i \lambda_i}}\times\Tw(f, 2^{\boldsymbol{\lambda}}\mathbf{u}),
\end{align*}
where $\Tw(f,\mathbf a)$ is as defined in~\eqref{eq:Tdef} and a vector being admissible aligns with Definition~\ref{def:admissible}.
\end{lemma}

\begin{proof}
First, we consider the factor at $\infty$. Applying~\cite[Proposition 4.1]{LRS}, we see that 
\[
\sigma_\infty(f) = \int_{\PPP^{n - 1}(\R)} e(\inv_\infty \alpha(x)) \mathrm{d}x = \frac{n}{2} \int_{\substack{ \mathbf x \in [-1,1]^n \\ \eqref{eq:def} \text{ soluble}}} e(\inv_\infty \alpha(\mathbf x)) \mathrm{d}\mathbf x.
\]
Hence we deduce from Lemma~\ref{lemma:invtoT} that
\[
\sigma_\infty(f) = \frac{n}{2} \sum_{\mathbf s \text{ admissable}} \Tw(f, \mathbf s).
\]
At 2, we may apply the same reasoning as for odd primes. Namely, 
\begin{align}
\label{eSigma2f}
\sigma_2(f) =\frac{\omega_2(\{x \in U(\Q_2)_\sB : \inv_2 \alpha(x) = 0\}) - \omega_2(\{x \in U(\Q_2)_\sB : \inv_2 \alpha(x) = \frac 12\})}{2^{1-n +\sum_i \vert V_{\{i\}}\vert^{-1}}}.
\end{align}
By~\cite[Proposition 4.1]{LRS} and Lemma~\ref{lemma:invtoT}, $\omega_2(\{x \in U(\Q_2)_\sB : \inv_2 \alpha(x) = c\})$ is equal to
\begin{multline*}
2 \left( 1 - \frac{1}{2^2}\right)^{-n} 
\sum_{ \boldsymbol{1}\neq \boldsymbol{\lambda} \in \{0,1\}^n} \left(1 - \frac{1}{4^{|\{i \in [n] : \lambda_i = 0\}|}}\right)\times \frac{1}{2^{\sum_i \lambda_i}}\\ \times \mu_2 \left(
\left\{
\mathbf u \in \left( \Z_2^\times\right)^n : 2^{\boldsymbol{\lambda}}\mathbf{u} \in U(\Q_2)_{\mathscr B},\ \Tw (f,2^{\boldsymbol{\lambda}}\mathbf{u}) = (-1)^{\iota(c)}
\right\}
\right)
\end{multline*}
where $\iota$ is the unique identification between $\frac{1}{2}\Z/\Z$ and $\Z/2\Z$. Since the Hilbert symbols are determined by congruence classes modulo $8$, and noting that $\mu_2(\{\mathbf v \in \left( \Z_2^\times\right)^n: \mathbf v \equiv \mathbf u \bmod 8\}) = 8^{-n}$, we rewrite the denominator of~\eqref{eSigma2f} as
$$
2^{1-n +\sum_i \vert V_{\{i\}}\vert^{-1}} \sigma_2(f) = 2 \left(\frac{4}{3}\right)^n 
\sum_{(\mathbf u, \boldsymbol \lambda) \text{ admissible}} \left(1 - \frac{1}{4^{|\{i \in [n] : \lambda_i = 0\}|}}\right) \times \frac{1}{2^{\sum_i \lambda_i}} \times \frac{\Tw(f, 2^{\boldsymbol{\lambda}}\mathbf{u})}{8^n},
$$
completing the desired description.
\end{proof}

We can gather together the work in this section so far to give an explicit description of the Loughran--Rome--Sofos constant.

\begin{corollary}
\label{cLRS}
We have
\begin{multline*}
c_{\LRS}
= \frac{n^{n- \sum_i\vert V_{\{i\}} \vert^{-1}}}{2^{1 + \sum_i \vert V_{\{i\}} \vert^{-1} }}\times \frac{1 }{3^n \prod_{i = 1}^n \Gamma(|V_{\{i\}}|^{-1})}
\\
\sum_{\substack{f \in \FAlt \\ \eqref{eq:weirdcondition}}} \sum_{\substack{\mathbf s \in \{\pm 1\}^n \\ \mathbf s \textup{ admissible} }} \Tw(f,\mathbf s)
\sum_{\substack{\boldsymbol{\lambda} \in \{0,1\}^n \\ \boldsymbol{\lambda} \neq \boldsymbol{1}}}\frac{1}{2^{\sum_i \lambda_i} }
\left(1 - \frac{1}{4^{|\{i \in [n] : \lambda_i = 0\}|}}\right) \hspace{-0.3cm}
\sum_{\substack{\mathbf u \bmod 8\\(\mathbf u,\boldsymbol \lambda) \textup{ admissible}}} \hspace{-0.3cm}
\Tw(f, 2^{\boldsymbol{\lambda}}\mathbf{u})
\prod_{p > 2} \sigma_p(f),
\end{multline*}
with $\Tw(f, \mathbf{a})$ as defined in~\eqref{eq:Tdef} and $\sigma_p(f)$ as in Proposition~\ref{lTama}.
\end{corollary}

\subsection{Proof of Theorem~\ref{thm:main}}
Now we are ready to complete the proof of Theorem~\ref{thm:main}.

\begin{proof}[Proof of Theorem~\ref{thm:main}]
In order to compare the predicted main term from Conjecture~\ref{conj:LRS} with Theorem~\ref{tTechnicalHeart}, we first observe that for every point in projective space there are 2 representatives, $\mathbf x$ and $-\mathbf x$, as primitive integer vectors. For such a representative, the standard anticanonical height is given by $H(\mathbf x)=\max_i \vert x_i \vert^n$. Thus, by incorporating this into~\eqref{eq:finalnonproj} of Theorem~\ref{tTechnicalHeart}, and the observation that $f \in \FAlt$ only contributes to the main term if~\eqref{eq:weirdcondition} holds thanks to Lemma~\ref{lemma:projcancellation}, we arrive at 
\begin{align*}
&N^{\mathrm{proj}}(B) \sim \frac{N(B^{1/n})}{2} \sim \frac{n^{n - \sum_i |V_{\{i\}}|^{-1}} B (\log B)^{-n + \sum_i |V_{\{i\}}|^{-1}}}{\prod_i \Gamma(|V_{\{i\}}|^{-1}) \times 2^{2n+1 + \sum_i |V_{\{i\}}|^{-1}}} \times \frac{1}{(1 - 2^{-2})^n} \times \hspace{-0.3cm}
\sum_{\substack{f \in \RAlt \\ J \subseteq \{i \in [n] : \{i\} \in \mathcal{D}\}\\ \eqref{eq:weirdcondition}}}\\
&\sum_{\mathbf{s} \text{ admissible}} \Tw(f, \mathbf{s}) \prod_{i \in J} \leg{-1}{s_i} \hspace{-0.1cm}
\sum_{(\mathbf{u}, \boldsymbol{\lambda}) \text{ admissible}} \hspace{-0.1cm}
\left(1 - \frac{1}{2^{2(n-\sum_i \lambda_i)}}\right) 
\frac{\Tw(f, 2^{\boldsymbol{\lambda}} \mathbf{u})}{2^{\sum_i \lambda_i}} \prod_{i \in J} \leg{-1}{u_i} \prod_{p>2} \Lambda_p(f,J).
\end{align*}
A number of the factors are immediately identifiable. Indeed, the order of magnitude, the Gamma factors and the factor of $n$ are all the same.

Next consider the Euler product term $\Lambda_p(f, J)$ for odd primes. Given an $f \in \RAlt$ and $J \subseteq \{i \in [n] :\{i \} \in \mathcal{D}\}$, we can associate an element $g \in \FAlt$ by~\eqref{eq:Altses}. We want to show that $\Lambda_p(f, J) = \sigma_p(g)$ if the pair $(f, J)$ corresponds to $g$. It is clear from the definition of $\RAlt$ that the condition $f(S) \in W_S$ is equivalent to $g(S) \in V_S +\langle \{-\}\rangle$, so it remains to verify that
\begin{equation}
\label{eq:checkbrm}
\BrM(f, J,S, p) = \leg{-1}{p}^{\kappa(g, S)}
\end{equation}
holds for all $\varnothing \subseteq S \subset [n]$. Recall from~\eqref{eq:kappaS} that
\[
\kappa(g,S)=\sum_{k \in S} \langle \{-\}, g(\{k\})\rangle + \sum_{\{i,j\}\subseteq S} \langle \{i\}, g(\{j\}) \rangle+ \langle \{-\}, f(S)\rangle.
\]
Then~\eqref{eq:checkbrm} is a consequence of the second part of Lemma~\ref{lemma:projqs}. Therefore the local factor $\Lambda_p(f, J)$ matches the conjecture at odd primes $p$.

It remains to examine the $2$-adic and real factors. Again we can identify a choice of $f \in \RAlt$ and $J$ with a map $g \in \FAlt$. Upon doing so, we apply the first part of Lemma~\ref{lemma:projqs} and we obtain the identity
$$
\Tw(f, \mathbf{a}) \prod_{i \in J} (-1, a_i)_2 = \Tw(g, \mathbf{a}).
$$
Therefore, the Hilbert symbol factors in both expressions are precisely the same (using that $(-1, 2)_2 = +1$). Thus all factors are accounted for and we may conclude that the leading constant in Theorem~\ref{thm:main} is precisely the one predicted by Loughran--Rome--Sofos.
\end{proof}

\section{R\'edei symbol application}
In this section we prove Theorem~\ref{tRedei}. Concretely, our goal is to obtain an asymptotic formula for the number of triples $(a, b, c)$ for which the R\'edei symbol $[a, b, c]$ takes the value $1$ or $-1$, see~\cite[Example 2.5]{KS} for the definition of this symbol. 
Recall the definition of the set $S(B)$ from~\eqref{eq:redeisetdef}. Using the trilinear large sieve in~\cite[Remark 3.16]{KS}, it follows that 
\[
\sum_{(a, b, c) \in S(B)} [a, b, c] \ll B^{3 - \frac{1}{512}} (\log B)^{1792}
\]
so for any given $\delta \in \{\pm 1\}$, we have
\begin{equation}
\label{eTrilinearSieve}
\left\lvert\left\{(a, b, c) \in S(B) : [a, b, c] = \delta\right\}\right\rvert = \frac{1}{2}|S(B)| + O\left(B^{3 - \frac{1}{512}} (\log B)^{1792}\right).
\end{equation}
It remains to evaluate $|S(B)|$. We return to the set-up of Section~\ref{sChar}. With the notation of that section, we take $n = m = 3$ and 
\begin{alignat*}{3}
&M_{1, 1}(\mathbf{t}) = -1, \quad &&M_{1, 2}(\mathbf{t}) = t_1, \quad && M_{1, 3}(\mathbf{t}) = t_2 \\ 
&M_{2, 1}(\mathbf{t}) = -1, \quad &&M_{2, 2}(\mathbf{t}) = t_1, \quad && M_{2, 3}(\mathbf{t}) = t_3 \\
&M_{3, 1}(\mathbf{t}) = -1, \quad &&M_{3, 2}(\mathbf{t}) = t_2, \quad && M_{3, 3}(\mathbf{t}) = t_3.
\end{alignat*}
This leads to the following table.
\begin{center}
\begin{tabular}{|c||c|}
\hline
$S$ & $V_S=W_S$ \\ \hline\hline
$\{1\}$ & $\{\varnothing,\{2\},\{3\},\{2,3\}\}$ \\ \hline
$\{2\}$ & $\{\varnothing,\{1\},\{3\},\{1,3\}\}$ \\ \hline
$\{3\}$ & $\{\varnothing,\{1\},\{2\},\{1,2\}\}$ \\ \hline
$\{2,3\}$ & $\{\varnothing,\{1\},\{-,2,3\},\{-,1,2,3\}\}$ \\ \hline
$\{1,3\}$ & $\{\varnothing,\{2\},\{-,1,3\},\{-,1,2,3\}\}$ \\ \hline
$\{1,2\}$ & $\{\varnothing,\{3\},\{-,1,2\},\{-,1,2,3\}\}$ \\ \hline
\end{tabular}
\end{center}
The vector space $\RAlt$ contains $f_\varnothing, f_1, f_2, f_3, f_{12}, f_{13} ,f_{23}, f_{123}$ defined as follows.
\begin{center}\resizebox{\columnwidth}{!}{
\begin{tabular}{|c||c|c|c|c|c|c|c|c|}
\hline
$S$ & $f_{\varnothing}(S)$ & $f_1(S)$ & $f_2(S)$ & $f_3(S)$ & $f_{23}(S)$ & $f_{13}(S)$ & $f_{12}(S)$ & $f_{123}(S)$\\ \hline\hline
$\{1\}$ & $\varnothing$& $\varnothing$ & $ \{3\}$ &$\{2\}$ & $\{2,3\}$ & $ \{2\}$ & $ \{3\}$& $\{2,3\}$ \\ \hline
$\{2\}$& $\varnothing$ & $\{3\}$ &$\varnothing$ & $\{1\}$ & $\{1\}$ & $ \{1,3\}$ & $ \{3\}$& $\{1,3\}$ \\ \hline
$\{3\}$& $\varnothing$ & $\{2\}$ & $\{1\}$ & $\varnothing$& $\{1\}$ & $ \{2\}$ & $ \{1,2\}$ & $\{1,2\}$\\ \hline
$\{2,3\}$& $\varnothing$ & $\{-, 2,3\}$ & $\{1\}$ & $\{1\}$& $\varnothing$ & $ \{-,1,2,3\}$ & $ \{-,1,2,3\}$ & $\{-,2,3\}$\\ \hline
$\{1,3\}$& $\varnothing$ & $\{2\}$ & $\{-,1,3\}$ & $\{2\}$& $\{-1,2,3\}$ & $\varnothing$ & $ \{-,1,2,3\}$ & $\{-,1,3\}$\\ \hline
$\{1,2\}$& $\varnothing$ & $\{3\}$ & $\{3\}$ & $\{-,1,2\}$& $\{-1,2,3\}$ & $\{-,1,2,3\}$ & $\varnothing$ & $\{-,1,2\}$\\ \hline
\end{tabular}}
\end{center}
Also note that $\mathcal{D}$ is empty, so $\Jay = 0$.

Consider the set
\[
S(B, \mathbf{s}, \mathbf{u}, \boldsymbol{\lambda}) \coloneqq \left\{(t_1, t_2, t_3) \in S(B) : 
\begin{array}{l}
\sgn(t_i) = s_i, \ \nu_2(t_i) \equiv \lambda_i \bmod 2 \\ t_i/2^{\lambda_i} \equiv u_i \bmod 8
\end{array}
\right\}.
\]
The condition $\gcd(\Delta(t_1), \Delta(t_2), \Delta(t_3)) = 1$ implies that $\gcd(t_1, t_2, t_3) = 1$ and that at least one of $t_1 ,t_2, t_3$ is congruent to $1 \bmod 4$. The possible assignments of $\mathbf{s}$ are $(1,1,1)$, $(-1,1,1)$, $(1,-1,1)$, $(1,1,-1)$. The possible assignments $2^{\boldsymbol{\lambda}}\mathbf{u}$, given that at least one of $t_1, t_2, t_3$ is congruent to $1 \bmod 8$, are permutations of 
\begin{gather*}
(1, 1, \pm 1) \bmod 4, \quad (1, \pm 1, 2) \bmod 8, \quad (1, 3, -2) \bmod 8, \quad (1, 1, -2) \bmod 8 \\
(1, -6, \pm 6) \bmod (8, 16, 16), \quad (1, 2, \pm 2) \bmod (8, 16, 16), \quad (1, -2, 6) \bmod (8,16,16).
\end{gather*}
By~\eqref{eFinalSum3} with $\mathbf{b} = (1, 1, 1)$, noting that $f(S) \in W_S$ is always satisfied and recalling that $\mathcal{D}$ is empty, we get
\begin{multline*}
\lvert S(B, \mathbf{s}, \mathbf{u}, \boldsymbol{\lambda})\rvert
=
\frac{B^3 (\log B)^{-3 + \sum_i |V_{\{i\}}|^{-1}} \times \left( 1 + O \left( (\log B)^{-1}\right) \right)}{2^{6 + \sum_i \left(\lambda_i + |V_{\{i\}}|^{-1}\right)} \times \prod_i \Gamma(|V_{\{i\}}|^{-1})} \prod_{p>2}\frac{\left( 1- p^{-1}\right)^{\sum_i|V_{\{i\}}|^{-1}}}{\Sigma(p)} \times \\
\sum_{f \in \RAlt } \Tw(f,\mathbf{tw}) \Tw(f,\mathbf{s}) \sum_{(w_{S})_{|S| \geq 2}} \frac{ g\left(\prod_{|S| \geq 2} w_S ,\mathbf{1}\right) \prod_{|S| \geq 2} \Br(f, S, w_S)}{\prod_{|S| \geq 2} c_S^{\omega(w_S)} \prod_{|S| \geq 2} w_S^{|S|} }.
\end{multline*}
Conveniently, we have $\Tw(f,\mathbf{tw}) = \Tw(f,\mathbf{s}) = 1$ for all admissible $ \mathbf{s}$ and $(\mathbf{u}, \boldsymbol{\lambda})$. We rewrite the inner sum as
\begin{multline*}
\prod_{p > 2} \left(\sum_{\substack{(j_S)_{|S| \geq 2} \\ 0 \leq \sum_{|S| \geq 2} j_S \leq 1}} \frac{g\left(\prod_{|S| \geq 2} p^{j_S} ,\mathbf{1}\right)\prod_{|S| \geq 2} \Br(f, S, p^{j_S})}{\prod_{|S| \geq 2} c_S^{\mathbf{1}_{j_S > 0}} p^{\sum_{|S| \geq 2} |S| j_S} }
\right)\\
= \prod_{p > 2} \left( 1+\sum_{|S| \geq 2} \frac{ g(p ,\mathbf{1}) \Br(f, S, p)}{c_S p^{|S| } }
\right).
\end{multline*}
Observe that
\[
g(p ,\mathbf{1}) = \Sigma(p) = \left(1 + \frac{3}{4p}\right)^{-1}.
\]
When $|S| = 2$, we can check that for every $f \in \RAlt$, we have 
$$
|\{-\} \cap f(S)| + \sum_{i \in S} |\{-\} \cap f(\{i\})| + \sum_{\{i, j\} \subseteq S} |\{j\} \cap f(\{i\})|\equiv 0\bmod 2,
$$ 
so $\Br(f, S, p) = 1$. Therefore we conclude that
\[
\lvert S(B, \mathbf{s}, \mathbf{u}, \boldsymbol{\lambda})\rvert
=
\frac{B^3 (\log B)^{-3 +\frac{3}{4}} \times \left( 1 + O \left( (\log B)^{-1}\right) \right)}{2^{3 + \sum_i \lambda_i +\frac{3}{4}} \times \Gamma(4^{-1})^3} \times 
\prod_{p > 2} \left(1- \frac{1}{p}\right)^{\frac{3}{4}}\left(1 + \frac{3}{4p} + \frac{3}{4 p^2} \right).
\]
Next we compute the possible number of lifts to $2^{\boldsymbol{\lambda}} \mathbf{u}$.
\begin{center}
\begin{tabular}{|c||c|c|c|}
\hline
$2^{\boldsymbol{\lambda}} \mathbf{u}$ up to permutation&\makecell{Total $\#2^{\boldsymbol{\lambda}} \mathbf{u}$\\ of this type}& $\sum_{i}\lambda_i$& $\frac{\#2^{\boldsymbol{\lambda}} \mathbf{u}}{2^{\sum_{i}\lambda_i}}$
\\ \hline\hline
$\{1,1,\pm 1\},\{1,1,\pm 3\},\{1,-1,-3\},\{1,-3,\pm 3\}$ & $25$& $0$& $25$ \\ \hline
$\makecell{\{1,\pm 1,2\},\{1, 3,-2\},\{1,1,-2\}, \\ \{1,\pm 1,-6\},\{1, 3,6\},\{1,1,6\}}$ & $36$ & $1$ & $18$ \\ \hline
$\{1,2,\pm 2\},\{1,-6,\pm 6\},\{1,-2,6\}$ & $24$ & $2$& $6$ \\ \hline
\end{tabular}
\end{center}
Finally, summing over the four possible $\mathbf{s}$, and over all $\mathbf{u}, \boldsymbol{\lambda}$, we have
\[
\lvert S(B)\rvert
= 49\times
\frac{B^3 \times \left( 1 + O \left( (\log B)^{-1}\right) \right)}{(\log B)^{\frac{9}{4}} \times 2^{\frac{7}{4}} \times\Gamma(4^{-1})^3} \times 
\prod_{p>2} \left( 1- \frac{1}{p}\right)^{\frac{3}{4}}\left( 1 + \frac{3}{4p}+\frac{ 3}{4 p^2} \right),
\]
which gives Theorem~\ref{tRedei} upon combining this with equation~\eqref{eTrilinearSieve}.

\bibliographystyle{abbrv}

\begin{thebibliography}{39}

\bibitem{AK}
B.~Alberts and J.~Klys.
\newblock The Distribution of $H_8$-Extensions of Quadratic Fields.
\newblock \textit{Int. Math. Res. Not.} {\bf 2021} (2021), no. 2, 1508--1572.

\bibitem{chats}
R.~de la Bret{\`e}che, T.~D.~Browning and E.~Peyre.
\newblock On Manin's conjecture for a family of Ch{\^a}telet surfaces.
\newblock \textit{Ann. of Math.} (2) \textbf{175} (2012), no. 1, 297--343.

\bibitem{BC}
T.~D.~Browning and S.~Chan.
\newblock Almost all quadratic twists of an elliptic curve have no integral points.
\newblock \textit{Preprint}, arXiv:2401.04375.

\bibitem{BHB21}
T.~D.~Browning and D.~R.~Heath-Brown.
\newblock The geometric sieve for quadrics.
\newblock \textit{Forum Math.} {\bf 33} (2021), no. 1, 147--165.

\bibitem{arXiv:2203.06881}
T.~D.~Browning, J.~Lyczak and R.~Sarapin.
\newblock Local solubility for a family of quadrics over a split quadric surface.
\newblock \textit{Involve} {\bf 16} (2023), no.~2, 331--342.

\bibitem{BLS}
T.~D.~Browning, J.~Lyczak and A.~Smeets.
\newblock Paucity of rational points on fibrations with multiple fibres.
\newblock \textit{Preprint}, arXiv:2310.01135.  To appear in \textit{Algebra Number Theory}.

\bibitem{Chan}
S.~Chan.
\newblock The 3-Isogeny Selmer Groups of the Elliptic Curves $y^2 = x^3 + n^2$.
\newblock \textit{Int. Math. Res. Not.} {\bf 2024} (2024), no. 9, 7571--7593.

\bibitem{Brauerbook}
J.-L.~Colliot-Th{\'e}l{\`e}ne and A.~N.~Skorobogatov. 
\newblock \textit{The Brauer--Grothendieck group.}
\newblock Ergebnisse der Mathematik und ihrer Grenzgebiete, 3. 
\newblock Folge, Band 71, Springer, 2021.

\bibitem{DS}
K.~Destagnol and E.~Sofos. 
\newblock Averages of arithmetic functions over polynomials in many variables.
\newblock \textit{Preprint}, arXiv:2409.18116.

\bibitem{DLS}
K.~Destagnol, J.~Lyczak and E.~Sofos. 
\newblock Local solubility in generalised Ch{\^a}telet surfaces.
\newblock \textit{Preprint}, arXiv:2504.11388.

\bibitem{FK4rank}
\'E.~Fouvry and J.~Kl\"uners.
\newblock On the 4-rank of class groups of quadratic number fields.
\newblock {\em Invent. Math.} {\bf 167} (2007), no.~3, 455--513.

\bibitem{FK}
\'E.~Fouvry and P.~Koymans.
\newblock Malle's conjecture for nonic Heisenberg extensions.
\newblock \textit{Preprint}, arXiv:2102.09465. To appear in \textit{Ann. Inst. Fourier}.

\bibitem{FKP}
\'E.~Fouvry, P.~Koymans and C.~Pagano.
\newblock On the 4-rank of class groups of Dirichlet biquadratic fields.
\newblock \textit{J. Inst. Math. Jussieu} {\bf 21} (2022), no. 5, 1543--1570.

\bibitem{FI}
J.~Friedlander and H.~Iwaniec.
\newblock Ternary quadratic forms with rational zeros.
\newblock \textit{J. Th\'eor. Nr. Bordx.} {\bf 22} (2010) no. 1, 97--113. 

\bibitem{Guo95}
C.~R.~Guo. 
\newblock On solvability of ternary quadratic forms.
\newblock \emph{Proc. London Math. Soc.} (3) {\bf 70} (1995), no. 2, 241--263.

\bibitem{HB}
D.~R.~Heath-Brown.
\newblock The size of Selmer groups for the congruent number problem. 
\newblock \emph{Invent. Math.} {\bf 111} (1993), no. 1, 171--196.

\bibitem{HB2}
D.~R.~Heath-Brown.
\newblock The size of Selmer groups for the congruent number problem, II. 
\newblock \emph{Invent. Math.} {\bf 118} (1994), no. 2, 331--370.

\bibitem{sumsofsquares}
D.~R.~Heath-Brown.
\newblock Linear relations amongst sums of two squares. 
\newblock In collection: \textit{Number theory and algebraic geometry}, p. 133--176.
London Math. Soc. Lecture Note Ser., \textbf{303}
Cambridge University Press, Cambridge, 2003.

\bibitem{Hoo93}
C.~Hooley. 
\newblock On ternary quadratic forms that represent zero.
\newblock \emph{Glasgow Math. J.} {\bf 35} (1993), no. 1, 13--23.

\bibitem{Hoo07}
C.~Hooley. 
\newblock On ternary quadratic forms that represent zero. II. 
\newblock \emph{J. reine angew. Math.} {\bf 602} (2007), 179--225.

\bibitem{Klys1}
J.~Klys.
\newblock The distribution of $p$-torsion in degree $p$ cyclic fields.
\newblock \textit{Algebra Number Theory} {\bf 14} (2020), no. 4, 815--854. 

\bibitem{Klys2}
J.~Klys.
\newblock Moments of unramified 2-group extensions of quadratic fields.
\newblock \textit{Preprint}, arXiv:1710.00793.

\bibitem{Kou}
D.~Koukoulopoulos.
\newblock \textit{The distribution of prime numbers},
\newblock Grad. Stud. Math., 203.
\newblock American Mathematical Society, Providence, RI, 2019.

\bibitem{KMS}
P.~Koymans, A.~Morgan and H.~Smit.
\newblock The 4-rank of class groups of $K(\sqrt{n})$.
\newblock \textit{Preprint}, arXiv:2101.03407.

\bibitem{KPSS}
P.~Koymans, R.~Paterson, T.~Santens and A.~Shute.
\newblock Local solubility of generalised Fermat equations.
\newblock \textit{Preprint}, arXiv:2501.17619.

\bibitem{KR}
P.~Koymans and N.~Rome.
\newblock Weak approximation on the norm one torus.
\newblock \textit{Compos. Math.} {\bf 160} (2024), no. 6, 1304--1348.

\bibitem{KS}
P.~Koymans and A.~Smith.
\newblock Sums of rational cubes and the $3$-Selmer group.
\newblock \textit{Preprint}, arXiv:2405.09311.

\bibitem{Lou18}
D.~Loughran.
\newblock The number of varieties in a family which contain a rational point.
\newblock \textit{J. Eur. Math. Soc.} {\bf 20} (2018), no. 10, 2539--2588.

\bibitem{LM}
D.~Loughran and L.~Matthiesen.
\newblock Frobenian multiplicative functions and rational points in fibrations.
\newblock \textit{J. Eur. Math. Soc.} {\bf 26} (2024), no. 12, 4779--4830.

\bibitem{LRS}
D.~Loughran, N.~Rome and E.~Sofos.
\newblock The leading constant for rational points in families.
\newblock \textit{Preprint}, arXiv:2210.13559.

\bibitem{LSBG}
D.~Loughran and T.~Santens.
\newblock Malle's conjecture and Brauer groups of stacks.
\newblock \textit{Preprint}, arXiv:2412.04196.

\bibitem{LS16}
D.~Loughran and A.~Smeets. 
\newblock Fibrations with few rational points.
\newblock \textit{Geom. Funct. Anal.} {\bf 26} (2016), no. 5, 1449--1482.

\bibitem{LTBT20}
D.~Loughran, R.~Takloo-Bighash and S.~Tanimoto.
\newblock Zero-loci of Brauer group elements on semi-simple algebraic groups.
\newblock \textit{J. Inst. Math. Jussieu} {\bf 19} (2020), no. 5, 1467--1507.

\bibitem{MP}
A.~Morgan and R.~Paterson.
\newblock On 2-Selmer groups of twists after quadratic extension.
\newblock \textit{J. London Math. Soc.} {\bf 105} (2022), 1110--1166.

\bibitem{Mori}
M.~Morishita.
\newblock \textit{Knots and primes: An introduction to arithmetic topology.} 
\newblock Universitext, Springer, London, 2012. 

\bibitem{PS}
C.~Pagano and E.~Sofos.
\newblock Diophantine stability and second order terms.
\newblock \textit{Preprint}, arXiv:2409.12144.

\bibitem{Peyre}
E.~Peyre.
\newblock Hauteurs et mesures de Tamagawa sur les vari{\'e}t{\'e}s de Fano.
\newblock \textit{Duke Math. J.} \textbf{79} (1995), no. 1, 101--218.

\bibitem{Redei}
L.~R\'edei. 
\newblock Ein neues zahlentheoretisches Symbol mit Anwendungen auf die Theorie der quadratischen Zahlk\"orper. I. 
\newblock \textit{J. Reine Angew. Math.} {\bf 180} (1939), 1--43.

\bibitem{Rome}
N.~Rome.
\newblock The Hasse Norm Principle For Biquadratic Extensions.
\newblock \textit{J. Th\'eor. Nr. Bordx.}, {\bf 30} (2018) no.31, 947--964.

\bibitem{Santens}
T.~Santens.
\newblock Diagonal quartic surfaces with a Brauer--Manin obstruction.
\newblock \textit{Compos. Math.} {\bf 159} (2023), no. 4, 659--710.

\bibitem{Ser90}
J.-P.~Serre.
\newblock Sp\'{e}cialisation des \'{e}l\'{e}ments de $\Br_2(\Q(T_1,\ldots,T_n))$.
\newblock \textit{C. R. Acad. Sci. Paris S\'{e}r. I Math.} {\bf 311} (1990), no. 7, 397--402.

\bibitem{Smith}
A.~Smith.
\newblock Governing fields and statistics for 4-Selmer groups and 8-class groups.
\newblock \textit{Preprint}, arXiv:1607.07860.

\bibitem{Smi22b}
A.~Smith.
\newblock The distribution of $\ell^\infty$-Selmer groups in degree $\ell$ twist families II.
\newblock \textit{Preprint}, arXiv:2207.05143.

\bibitem{MR4353917}
E.~Sofos and E.~Visse-Martindale.
\newblock The density of fibres with a rational point for a fibration over hypersurfaces of low degree.
\newblock \textit{Ann. Inst. Fourier} {\bf 71} (2021), no. 2, 679--709.

\bibitem{Stevenhagen}
P.~Stevenhagen.
\newblock Redei reciprocity, governing fields, and negative Pell.
\newblock \textit{Math. Proc. Camb. Philos. Soc.} {\bf 172} (2022), no. 3, 627--654.

\bibitem{cameron}
C.~Wilson.
\newblock Asymptotics for local solubility of diagonal quadrics over a split quadric surface.
\newblock \textit{Preprint}, arXiv:2404.11489.
\end{thebibliography}

\end{document}